%% file: main.tex
\documentclass[11pt]{article}
\usepackage{graphicx}
\usepackage{amsthm, amsmath, amssymb, tikz, bm}
\usepackage{mathtools}
\usepackage{xifthen}
\usepackage{comment}

\usepackage[margin=1.0in]{geometry} 
\usepackage{thm-restate}
\usepackage[shortlabels]{enumitem}
\usepackage{todonotes}
\usepackage[colorlinks=true, linkcolor=blue,citecolor=blue,
urlcolor=blue]{hyperref}

\usetikzlibrary{calc,shapes, backgrounds}
\usetikzlibrary[patterns]
\usetikzlibrary{arrows.meta}
\usetikzlibrary{decorations.markings}

\newtheorem{theorem}{Theorem}
\newtheorem{lemma}[theorem]{Lemma}

\newtheorem{claim}{Claim}

\newtheorem{observation}{Observation}

\newcommand{\lp}{\left (}
\newcommand{\rp}{\right )}

\DeclarePairedDelimiter{\floor}{\lfloor}{\rfloor}

\makeatletter
               {\list{}{\leftmargin=0pt 
                        \labelwidth\z@ \itemindent-\leftmargin
                        }}%
               {\endlist}
\makeatother

\title{Counting $k$-cycles in $5$-connected planar triangulations}

\author{
Gyaneshwar Agrahari \thanks{Department of Mathematics, Louisiana State University, Baton Rouge, LA, 70803 ({\tt gagrah1@lsu.edu}).}
\and
Xiaonan Liu \thanks{Department of Mathematics, Vanderbilt University, Nashville, TN, 37240
({\tt xiaonan.liu@vanderbilt.edu}).}
\and
Zhiyu Wang \thanks{Department of Mathematics, Louisiana State University, Baton Rouge, LA, 70803
({\tt zhiyuw@lsu.edu}). This author was supported in part by LA Board of Regents grant LEQSF(2024-27)-RD-A-16.}
}
\begin{document}

\maketitle

\begin{abstract}
We show that every $n$-vertex $5$-connected  planar triangulation has at most $9n-50$ many cycles of length $5$ for all $n\ge 20$ and this upper bound is tight.  We also show that for every $k\geq 6$, there exists some constant $C(k)$ such that for sufficiently large $n$, every  $n$-vertex $5$-connected planar graph has at most $C(k) \cdot n^{\lfloor k/3 \rfloor}$ many cycles of length $k$. This upper bound is asymptotically tight for all $k\geq 6$.
\end{abstract}

\section{Introduction}\label{sec:intro}

Enumerating cycles in a graph has a long history and has considerable applications in, for example, coding theory~\cite{AACST2013, AACST2013b, AACST2015}. In this paper, we focus on enumerating cycles in $5$-connected planar triangulations. A graph G is \textit{$k$-connected} if it has more than $k$ vertices and remains connected when fewer than $k$ vertices are removed. A \textit{planar triangulation} is an edge-maximal plane graph with at least three vertices, i.e., every face is bounded by a triangle. A cycle $C$ in a graph $G$ is said to be \textit{separating} if the graph obtained from $G$ by deleting vertices in $C$ is not connected.

The study on the existence and enumeration of Hamiltonian cycles (i.e., spanning cycles) in planar graphs has a long history~\cite{Whitney1931, Moon-Moser1963, Tutte1956, HST1979}, dating back to Whitney \cite{Whitney1931} who showed in 1931 that every $4$-connected planar triangulation has a Hamiltonian cycle. In 1979, Hakimi, Schmeichel, and Thomassen~\cite{HST1979} conjectured that every  $n$-vertex $4$-connected planar triangulation $G$ has at least $2(n-2)(n-4)$ Hamiltonian cycles, which was recently confirmed asymptotically by Liu, Wang, and Yu~\cite{Liu-Wang-Yu2021}. See also~\cite{BSC2018, Lo2020, Liu-Yu2021, LQ2021, Liu-Wang-Yu2021} for partial progress on this problem. For $5$-connected planar triangulations, 
the number of Hamiltonian cycles is drastically different. In particular, Alahmadi, Aldred, and Thomassen \cite{AAT2020} showed that every  $n$-vertex $5$-connected planar triangulation has $2^{\Omega (n)}$ Hamiltonian cycles, in contrast to the asymptotically tight quadratic lower bound when the connectivity is $4$. See~\cite{BHT1999, Liu-Yu2021} for related results. For more results on the minimum number of cycles of various length in planar triangulations, see \cite{LZ2025} and the references within.

There have also been extensive studies on enumerating cycles of fixed length in planar graphs. Let $N(H,n)$ denote the maximum number of copies of $H$ in an $n$-vertex planar graph. Alon and Caro~\cite{AC1984} determined $N(H,n)$ when $H$ is either a complete bipartite graph or a triangulation without separating triangles. In 1979, Hakimi and Schmeichel~\cite{HS1979} initiated the study on the maximum number of $k$-cycles in an $n$-vertex planar graph. In particular, they~\cite{HS1979} showed that $N(C_3,n) = 3n-8$, $N(C_4, n) = \frac{1}{2}(n^2 + 3n-22)$ and $N(C_k, n) = \Theta(n^{\floor{k/2}})$. Recently, answering a conjecture of Hakimi and Schmeichel~\cite{HS1979}, Gy\H{o}ri, Paulos, Salia, Tompkins and Zamora~\cite{GPSTZ} showed that $N(C_5, n) = 2n^2-10n+12$ (except for $n\in \{5,7\}$). Cox and Martin~\cite{Cox-Martin2022, Cox-Martin2023} further determined the exact asymptotics of $N(C_{2k},n)$ for $k\in \{3,4,5,6\}$. Confirming a conjecture of Cox and Martin~\cite{Cox-Martin2022}, Lv, Gy\H{o}ri, He, Salia, Tompkins and Zhu showed~\cite{Lv2024} that $N(C_{2k},n) = \frac{n^k}{k^k}+ o(n^k)$ for all $k\geq 3$. More recently, Heath, Martin and Wells~\cite{HMW2025} showed that $N(C_{2k+1},n)\leq 3k\lp \frac{n}{k}\rp^k + O(n^{k-1/5})$ for all $k\geq 5$, which is tight up to a factor of $3/2$. For general $H$, Wormald~\cite{Wormald1986} and later Eppstein~\cite{Eppstein1993} independently showed that for any $3$-connected graph $H$, $N(H,n) = \Theta(n)$. More results on generalized planar Tur\'an number of different graphs can be found in \cite{GPSTZ2}. 
The order of magnitude of $N(H,n)$ for an arbitrary graph $H$ was settled by Huynh, Joret and Wood~\cite{HJW2022}, and extended to more general settings when the host graphs are other classes of sparse graphs (see, e.g., \cite{Huynh-Wood2022, Liu2021+}).

In this paper, inspired by the drastically different behavior of the number of (Hamiltonian) cycles in planar graphs with different connectivities, we study the maximum number of $k$-cycles in $5$-connected planar graphs. Observe that it suffices to consider this problem in an edge-maximal planar graph, i.e., a planar triangulation. Morever, since an $n$-vertex planar triangulation has at most $3n-6$ edges, it can have connectivity at most $5$. 

For $\kappa\in \{3,4,5\}$ and a fixed graph $H$, let $N_{\kappa}(H,n)$ denote the maximum number of copies of $H$ in an $n$-vertex planar triangulation with connectivity $\kappa$.
For $3$-cycles, recall that Hakimi and Schmeichel~\cite{HS1979} showed $N(C_3,n) = 3n-8$. Since $4$-connected triangulations have no separating triangles, this implies that the number of $C_3$s in an $n$-vertex $4$-connected planar graph is at most $2n-4$, which is attained by any $n$-vertex planar triangulation. For $4$-cycles, Hakimi and Schmeichel~\cite{HS1979} showed that $N(C_4, n) = \frac{1}{2}(n^2 + 3n-22)$ for $n\geq 5$ and $N_4(C_4, n) = (n^2-n-2)/2$ for $n\geq 7$. Since a $5$-connected planar triangulation has no separating $4$-cycles, it follows that $N_5(C_4,n)$ is exactly the number of copies of $K_4-e$ in an  $n$-vertex $5$-connected planar triangulation, which is $3n-6$. For $5$-cycles, recall that Gy\H{o}ri, Paulos, Salia, Tompkins and Zamora~\cite{GPSTZ} showed that $N(C_5,n) = 2n^2-10n+12$ (for $n\notin\{5,7\}$), and the extremal graph is the double wheel, which is also $4$-connected (but not $5$-connected). This implies that $N_4(C_5,n) = 2n^2-10n+12$. The value of $N_5(C_5,n)$, even the order of asymptotics, was not known.

Our first result concerns the maximum number of $5$-cycles in a $5$-connected planar triangulation. In particular, we determine the exact value of $N_5(C_5,n)$ and also characterize the unique extremal graph.

\begin{theorem}\label{thm:5-cycle-5-connected}
For $n\geq 20$, every $n$-vertex $5$-connected planar triangulation has at most $9n-50$ many $5$-cycles, and the upper bound is tight. The unique extremal graph attaining the upper bound is the graph $D_1$ when $n$ is even and $D_2$ when $n$ is odd (see Figure \ref{fig:extremal}).
\end{theorem}

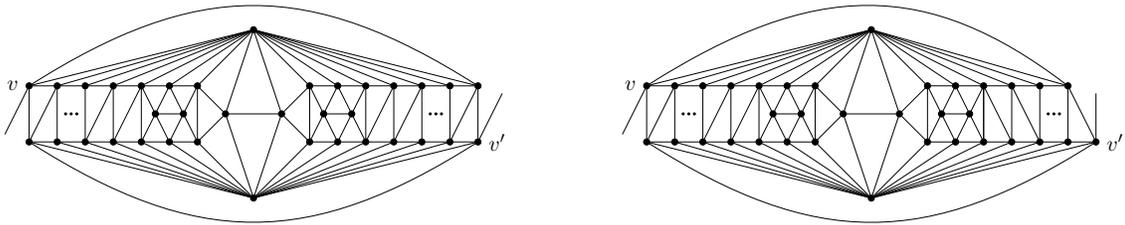
\begin{figure}[htb]
  \centering
  \begin{minipage}{.3\textwidth}
    \centering
     \resizebox{7cm}{!}{\input{5_connected_extremal.tikz}}
  \end{minipage}
  \hspace{3cm}
  \begin{minipage}{.4\textwidth}
    \centering
   \resizebox{7cm}{!}{\input{5_connected_extremal2.tikz}}
  \end{minipage}
  \caption{Extremal graphs $D_1$ and $D_2$ of Theorem \ref{thm:5-cycle-5-connected} when $n$ is even and odd respectively. The vertices $v$ and $v'$ are adjacent, as indicated by the two half lines.}
  \label{fig:extremal}
\end{figure}

For $n\leq 19$, the values of $N_5(C_5, n)$ are listed in Table~\ref{table:small_values}.
We remark that there does not exist any $5$-connected planar triangulation on fewer than $12$ vertices, and also on $13$ vertices.

\begin{table}[htb]\label{table:small_values}
    \centering
    \begin{tabular}{|c|c|c|}
        \hline
        $n$ & $N_5(C_5, n)$ & Maximum number of separating $5$-cycles \\
        \hline
        12  & 72  & 12   \\
        14 & 84  & 12   \\
        15 & 90  & 12   \\
        16  & 98  & 14 \\
        17  & 103  & 13    \\
        18  & 112  & 16     \\
        19  & 122  & 20   \\
        20 & 130 & 22\\
        \hline
    \end{tabular}
    \caption{Values of $N_5(C_5,n)$ for $n\leq 20$.}
    \label{tab:example}
\end{table}
For $k$-cycles with $k\geq 6$, recall that Hakimi and Schmeichel~\cite{HS1979} showed that $N(C_k,n)= \Theta(n^{\floor{k/2}})$. Moreover, the extremal graphs they provided can be easily modified to be $4$-connected. Hence, it follows that for every $k\geq 4$, $N_4(C_k,n) =\Theta(n^{\floor{k/2}})$. As seen from Theorem \ref{thm:5-cycle-5-connected}, the behavior of $N_5(C_k,n)$ can be significantly different due to the high connectivity constraint. Indeed, we show that for every $k\geq 6$, $N_5(C_k,n) = \Theta(n^{\floor{k/3}})$, in contrast to the fact that $N(C_k,n)= N_4(C_k,n) = \Theta(n^{\floor{k/2}})$.

\begin{theorem}\label{thm:long-cycles}
    For every $k\geq 6$, there exists some constant $C=C(k)$ such that for sufficiently large $n$, every $n$-vertex $5$-connected planar graph has at most $C\cdot n^{\floor{k/3}}$ many cycles of length $k$.
\end{theorem}

Recently in~\cite{GPSTZ2}, Gy\H{o}ri, Paulos, Salia, Tompkins and Zamora showed that the maximum number of $k$-cycles (for $k\geq 5$) in an $n$-vertex $C_4$-free planar graph is $O(n^{\floor{k/3}})$, and this upper bound is asymptotically tight. 
Using similar ideas of \cite{GPSTZ2}, we strengthen their result by showing the following theorem, which will imply Theorem \ref{thm:long-cycles}.

\begin{restatable}{theorem}{stronger}\label{thm:stronger}
     For every $k\geq 6$, if $G$ is an $m$-edge planar graph such that every two distinct vertices in $G$ have at most two common neighbors, then $G$ has at most    
     $\lp 342\cdot \frac{4^{\floor{k/3}}}{{\floor{k/3}}!}\rp \cdot m^{\floor{k/3}}$ many cycles of length $k$. 
\end{restatable}

Observe that in a $C_4$-free graph, every two distinct vertices have at most one common neighbor. Hence, Theorem \ref{thm:stronger} is a strengthening of the above upper bound of Gy\H{o}ri et~al~\cite{GPSTZ2}. However, apart from the use of Lemma \ref{lem:light-edges-min-degree-2}, our proof of Theorem \ref{thm:stronger} essentially follows along the same lines of Theorem $10$ in \cite{GPSTZ2}. We also remark that Lemma \ref{lem:light-edges-min-degree-2} can be adapted to all $K_{2,t}$-free planar graphs with minimum degree at least $2$ (with $t\geq 2$). Hence our proof of Theorem \ref{thm:stronger} implies that for any fixed integers $t\geq 2$ and $k\geq 5$, the maximum number of $k$-cycles in an $n$-vertex $K_{2,t}$-free planar graph is $O_{t,k}(n^{\floor{k/3}})$.

\begin{figure}[htb]
  \centering
  \begin{minipage}{.25\textwidth}
    \centering
     \resizebox{4.5cm}{!}{\input{diamond.tikz}}
  \end{minipage}
  \hspace{1cm}
  \begin{minipage}{.25\textwidth}
    \centering
   \resizebox{4.5cm}{!}{\input{blow_up.tikz}}
  \end{minipage}
  \hspace{1cm}
   \begin{minipage}{.25\textwidth}
    \centering
   \resizebox{4.5cm}{!}{\input{blow-up2.tikz}}
  \end{minipage}
  \caption{Diamond graph on $18$ vertices (left); a width-$5$ inflated blow-up of $C_4$ (middle); a $5$-connected supergraph of $B^{(5)}(C_4)$ (right): the outer face is triangulated similarly, but now shown for clarity's purpose.}
  \label{fig:diamond}
\end{figure}
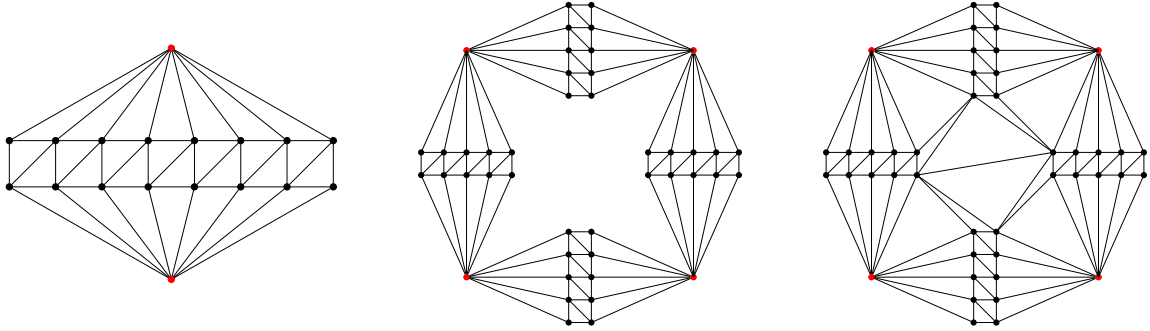

We remark that the upper bounds in Theorem \ref{thm:long-cycles} and Theorem \ref{thm:stronger} are asymptotically tight. 
Given an integer $\ell\geq 2$, let $F_1, F_2$ be two fan graphs, each on $\ell+1$ vertices, i.e., $F_1$ is obtained from joining a vertex $x_1$ with a path $P_1:= v_1 v_2 \cdots v_\ell$ on $\ell$ vertices, and similarly $F_2$ is obtained from joining a vertex $x_2$ with a path $P_2:= u_1 u_2 \cdots u_\ell$ on $\ell$ vertices.
Given an even integer $2\ell+2$ where $\ell\geq 2$, we define the \textit{diamond graph} on $2\ell+2$ vertices as the graph $D$ obtained from two fan graphs $F_1$, $F_2$, each on $\ell+1$ vertices, such that $v_i u_i\in E(D)$ for every $i\in [\ell]$, and $v_{j+1} u_j\in E(D)$ for every $j\in [\ell-1]$, see Figure \ref{fig:diamond}. 

For each integer $k\ge 6$ and sufficiently large $n$, we want to construct $n$-vertex $5$-connected  planar graphs with $\Omega(n^{\floor{k/3}})$ many $k$-cycles. Consider the following class of graphs. For $k\ge 9$, let $p:=\floor{k/3}$. Then $p\ge 3$.
Given a cycle $C_p$, we define the \textit{width-$w$ inflated blow-up} of $C_p$, denoted as $B^{(w)}(C_p)$, as the graph obtained from $C_p$ by replacing each edge of $C_p$ by the diamond graph on $2w+2$ vertices, see Figure \ref{fig:diamond}. Observe that $B^{(w)}(C_p)$ can be triangulated into a $5$-connected plane triangulation, see e.g., Figure \ref{fig:diamond} (right). Hence given any integer $p\geq 3$ and sufficiently large $n$, let $w:=\floor{(n-p)/2p}$ and we can construct an $n$-vertex $5$-connected planar triangulation containing $B^{(w)}(C_p)$ as subgraph. It is not hard to see that for $w=\floor{(n-p)/2p}$, $B^{(w)}(C_p)$ has at least $(2w)^p=\Omega(n^p)$ many $k$-cycles. For $6\leq k\leq 8$, the graph $D_1$ or $D_2$ (in Theorem \ref{thm:5-cycle-5-connected}) on $n$ vertices has $\Omega(n^2)$ many $6$-cycles, $7$-cycles and $8$-cycles. This shows that the upper bounds in Theorem \ref{thm:long-cycles} and Theorem \ref{thm:stronger} are asymptotically tight.

We also remark that Theorem \ref{thm:stronger} is stronger than Theorem \ref{thm:long-cycles}.
A $5$-connected planar graph clearly satisfies the assumption in Theorem \ref{thm:stronger}, as otherwise there will be a separating $4$-cycle. Moreover, since an $n$-vertex ($n\ge 3$) planar graph has at most $3n-6$ edges, Theorem \ref{thm:long-cycles} follows from Theorem \ref{thm:stronger} by setting $m=3n-6$. It will be interesting to determine the best possible coefficient of $n^{\floor{k/3}}$ and $m^{\floor{k/3}}$ in Theorem \ref{thm:long-cycles} and Theorem \ref{thm:stronger}.

\smallskip

\textbf{Organization and Notation.} We show some preliminary lemmas in Section \ref{sec:preliminaries}, show Theorem \ref{thm:5-cycle-5-connected} in Section \ref{sec:5-cycle} and show Theorem \ref{thm:stronger} in Section \ref{sec:k-cycle}.
Given a vertex $v\in V(G)$ in a graph $G$, we use $N_G(v)$ to denote the set of neighbors of $v$ in $G$, and use $d_G(v)$ to denote $|N_G(v)|$. We often ignore the subscript $G$ when it is clear from the context. Moreover, given a graph $G$, we use $N(C_k,G)$ to denote the number of $k$-cycles in $G$.

\section{Preliminaries}\label{sec:preliminaries}

In this section, we show a number of lemmas needed for the proof of Theorem \ref{thm:5-cycle-5-connected}.
A \textit{near triangulation} is a plane graph in which all faces except possibly its infinite face are bounded by triangles. For a cycle $C$ in a near triangulation $G$, we use $\overline{C}$ to denote the subgraph of $G$ consisting of all vertices and edges of $G$ contained in the closed disc in the plane bounded by $C$. The {\it interior} of $C$ in $G$, denoted by $\textrm{Int}_G(C)$, is then defined as the subgraph $\overline{C}-C$. 
Similarly, we define the \textit{exterior} of $C$ in $G$, denoted by $\textrm{Ext}_G(C)$, as the subgraph of $G$ induced by $V(G)\backslash V(\overline{C})$.

Let $W_6$ be the wheel graph on $6$ vertices, i.e., the graph obtained from $C_5$ by connecting a new vertex with every vertex of the $C_5$. Let $G_{11}$ be the graph as shown in Figure~\ref{fig:G11}. Given a $5$-cycle $C$ in a near triangulation $G$, we say $C$ is \textit{special} if $\textrm{Int}_G(C)$ has at least two vertices. For a separating $5$-cycle $C$ in a planar triangulation $G$, we say $C$ is \textit{non-trivial} if both $\textrm{Int}_G(C)$ and $\textrm{Ext}_G(C)$ have at least two vertices, and
\textit{trivial} if either $\overline{C}=W_6$ or $G-\textrm{Int}_G(C)=W_6$.

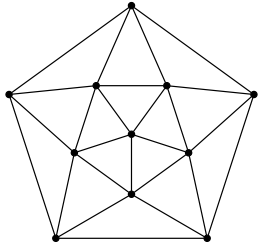
\begin{figure}[htb]
\begin{center}
      \resizebox{3.5cm}{!}{\input{G11.tikz}}
      \caption{The graph $G_{11}$.}
      \label{fig:G11}
\end{center} 
\end{figure}

\begin{observation}\label{obs:int_degree}
Let $G$ be an $n$-vertex near triangulation with no separating triangle or $4$-cycle. Suppose its outer cycle $C$ is a $5$-cycle and $n\ge 6$. Then the followings hold.
\begin{itemize}
\item [(i)] $C$ is a cycle with no chord, i.e., $G[V(C)]=C$, and every vertex in $C$ has degree at least three and every vertex in $G-C$ has degree at least five in $G$;
\item [(ii)] and if $C$ is the unique special $5$-cycle in $G$, every vertex in $C$ has degree at least four in $G$.
\end{itemize}
 \end{observation}
 \begin{proof}
 Suppose $G$ is an $n$-vertex near triangulation with no separating triangle or $4$-cycle such that its outer cycle $C=v_1v_2v_3v_4v_5v_1$ and $n\ge 6$. It follows that $G$ is not outerplanar as $n\ge 6$ and $|V(C)|=5$. Since $G$ has no separating triangle or $4$-cycle, we know that $v_i$ is not adjacent to any vertex in $V(C)\backslash N_C(v_i)$ for each $i\in [5]$ and this implies that $G[V(C)]=C$. Now we show that each vertex in $C$ must have a neighbor in $V(G-C)$. Suppose not. Without loss of generality, we may assume that $v_1$ has no neighbor in $V(G-C)$. Then $v_2v_5\in E(G)$ as $G$ is a near triangulation. It follows that $C$ has a cord, giving a contradiction. Hence $d(v)\ge 3$ for every $v\in V(C)$. For $v\in V(G-C)$, if $d(v)< 5$ then $d(v)=3$ or $d(v)=4$ and $G[N(v)]$ gives a separating triangle or $4$-cycle in $G$, a contradiction. Thus (i) holds. 

Now suppose $C$ is the unique special $5$-cycle in $G$. Since $C$ is special, it follows that $n\ge 7$. By (i), we know $d(v)\ge 3$ for each $v\in V(C)$. Assume for contradiction that $d(v)<4$ for some $v\in V(C)$. Without loss of generality, we may assume that $d(v_1)=3$ and let $v_1'$ be the unique neighbor of $v_1$ in $V(G-C)$. Since $G$ is a near triangulation, $v_1'$ is also adjacent to $v_2$ and $v_5$. Let $C':=v_1'v_2v_3v_4v_5v_1'$. Note that $\overline{C'}$ is a near triangulation on $n-1\ge 6$ vertices with no separating triangle or $4$-cycle and its outer cycle $C'$ is a $5$-cycle. Since $d_G(v_1')\geq 5$ and $C'$ is a cycle with no chord, we obtain that $v_1'$ has at least two neighbors in the interior of $C'$. It follows that the interior of $C'$  has at least two vertices and $C'$ is also a special $5$-cycle in $G$, a contradiction. Hence (ii) holds.
\end{proof}
\begin{lemma}\label{lem:minimal_5_cycle}
Let $G$ be an $n$-vertex near triangulation with no separating triangle or $4$-cycle. Suppose its outer cycle $C$ is a $5$-cycle and $n\ge 6$. Then one of the followings holds. 
\begin{itemize}
\item [(i)] $n=6$ and $G=W_6$;
\item [(ii)] $n\ge 11$. Moreover, if $n=11$ then $G=G_{11}$. 
\end{itemize}
\end{lemma}
\begin{proof}Suppose $G$ is an $n$-vertex near triangulation with no separating triangle or $4$-cycle such that its outer cycle $C=v_1v_2v_3v_4v_5v_1$ and $n\ge 6$. By Observation~\ref{obs:int_degree} (i), we know that $v_i$ is not adjacent to any vertex in $V(C)\backslash N_C(v_i)$, $v_i$ has at least one neighbor in $V(G-C)$ for each $i\in [5]$ and that $d(v)\ge 5$ for each $v\in V(G-C)$. If $n=6$, then it is clear that $G=W_6$. Now we assume $n\ge 7$.  We will show that $n\geq 11$. 

Suppose $G$ is vertex-minimum. This implies that $C$ is the unique special $5$-cycle in $G$. It follows from Observation~\ref{obs:int_degree} that for each $i\in [5]$, $v_i$ has degree at least four in $G$ and for every other vertex $v\in V(G-C)$, we have $d(v)\geq 5$.
Moreover, since $G$ is a near triangulation with its outer cycle being a $5$-cycle, we have that $|E(G)| = 3n-8$. It follows that 
\begin{equation}\label{eq:degree_sum}
 4\cdot 5 + 5\cdot (n-5) \leq \sum_{v\in V(C)} d(v) + \sum_{v\in V(G-C)} d(v) = \sum_{v\in V(G)} d(v)  = 2(3n-8),   
\end{equation}
which implies that $n\geq 11$. Furthermore, when $n=11$, the inequality in \eqref{eq:degree_sum} must be an equality. Hence we have $d(v_i) = 4$ for all $i\in [5]$, and $d(v)=5$ for all $v\in V(G-C)$. This implies that $G=G_{11}$. 
\end{proof}

\begin{observation}\label{obs:non-adjacent_path_length}
    Let $G$ be a near triangulation with its outer cycle $C:=v_1v_2v_3v_4v_5v_1$ being a $5$-cycle. Suppose that $G$ has no separating triangle or $4$-cycle and that $C$ is the unique special cycle in $G$. Then for any two non-adjacent vertices $u,v$ in $C$, every $uv$-path $P$ with $V(P)\backslash\{u,v\} \subseteq V(G-C)$ has length at least three.
\end{observation}
\begin{proof}
By Observation \ref{obs:int_degree}, $C$ is a cycle with no chord in $G$. Hence suppose for contradiction and without loss of generality that there exists some vertex $x \in V(G-C)$ such that $v_1 x v_3$ is a path of length two in $G$. Then $xv_2\in E(G)$ since $G$ is near triangulation and contains no separating $4$-cycle. This implies that $d(v_2)=3$, contradicting (ii) of Observation \ref{obs:int_degree}.
\end{proof}

\section{Number of $5$-cycles in $5$-connected planar triangulations}\label{sec:5-cycle}

In this section, we prove Theorem \ref{thm:5-cycle-5-connected}. For convenience, let $N_t(C_5,G)$ and $N_s(C_5, G)$ denote the number of trivial and non-trivial separating $5$-cycles in a plane triangulation $G$ respectively.

\begin{proof}[Proof of Theorem \ref{thm:5-cycle-5-connected}]
Let $G$ be a $5$-connected plane triangulation on $n$ vertices. We show Theorem \ref{thm:5-cycle-5-connected} by induction on $n$. The values of $N_5(C_5,n)$ for $n\leq 20$ are determined by computer search\footnote{The code can be accessed in \url{https://github.com/wzy3210/cycles_in_planar_triangulations}.}. We may now assume that $n\geq 21$. Observe that any $5$-cycle $C$ in $G$ is either a separating $5$-cycle, or a non-separating $5$-cycle. If it is non-separating, then either $\overline{C}$ or $G-\textrm{Int}_G(C)$ is outerplanar, i.e., $C$ is obtained from `gluing' two facial triangles on two sides of another facial triangle. It follows that the number of non-separating $5$-cycles is exactly
$(2n-4)\cdot \binom{3}{2} = 6n-12$. Therefore, to show Theorem \ref{thm:5-cycle-5-connected}, it suffices to show that the number of separating $5$-cycles in $G$ is at most $(9n-50)-(6n-12)=3n-38$.

If $G$ contains no non-trivial separating $5$-cycles, then we are done, as the number of trivial separating $5$-cycles, which is the number of degree $5$ vertices in $G$, is at most $n< 3n-38$ (since $n\ge 21$). Otherwise, let $C:= v_1v_2v_3v_4v_5v_1$ be a non-trivial separating $5$-cycle in $G$ such that $\overline{C}$ is vertex-minimum. Let $H:=\overline{C}$. By Lemma \ref{lem:minimal_5_cycle}, $|V(H)|\geq 11$, and $|V(H)|=11$ if and only if $H \cong G_{11}$. Note that $C$ is the unique special $5$-cycle in $H$ by the choice of $C$.\\

\textbf{Case 1}: Suppose $|V(H)|\geq 12$, i.e., $H\ne G_{11}$.  Let $G'$ be obtained from $G$ by replacing the interior of $C$ with the interior of $G_{11}$ such that $C\cup \textrm{Int}_{G'}(C)$ induces $G_{11}$ in $G'$ (i.e., replacing $H$ by $G_{11}$). Let $F=G'[C\cup \textrm{Int}_{G'}(C)]$ and then we know that $F$ is a near triangulation with outer cycle $C$ and $F\cong G_{11}$. Note that $|V(G')|\le n-1$ as $|V(H))|\ge 12>|V(F)|=11$. We claim that $G'$ is still $5$-connected. Suppose otherwise $G'$ has a separating triangle or $4$-cycle, say $D$. Since both $G$ and $F=G_{11}$ have no separating triangle or $4$-cycle, $D$ must use a vertex in $\text{Int}_{G'}(C)$ and a vertex in $\textrm{Ext}_{G'}(C)$. Thus $D$ must be a $4$-cycle; and we know that $D$ has to use two non-adjacent vertices in $C$, otherwise those two adjacent vertices in $C$ with a vertex in $\textrm{Int}_{G'}(C)$ or $\textrm{Ext}_{G'}(C)=\textrm{Ext}_{G}(C)$ is a separating triangle in $F=G_{11}$ or $G$. Note that $C$ is also the unique special $5$-cycle in $F$. It follows from Observation~\ref{obs:non-adjacent_path_length} that for any two non-adjacent vertices $u,v$ in $C$, there exists no $uv$-path of length two in $F$ such that all its internal vertices are contained in $V(F-C)$. Hence $D$ cannot exist, and it follows that $G'$ is $5$-connected. Let $n':=|V(G')|$. $G'$ has at most $3n'-38$ separating $5$-cycles  by induction if $n'\geq 20$, and note that $N_s(C_5,G')$ is at most the value that is listed in the third column of Table \ref{thm:5-cycle-5-connected} for each $n'\le 19$. Now we bound the number of separating $5$-cycles in $G$. We first bound the number of trivial separating $5$-cycles in $G$, i.e., the number of degree $5$-vertices in $G$. Observe that
$$|\{ v\in V(\textrm{Ext}_{G}(C)): d_G(v)=5\}| = |\{ v\in V(\textrm{Ext}_{G'}(C)): d_{G'}(v)=5\}|.$$
Since $C$ is the unique special $5$-cycle in $H$, every vertex in $V(C)$ has at least two neighbors in $V(H-C)$ by (ii) of Observation~\ref{obs:int_degree}. Note that every vertex in $V(C)$ has degree exactly four in $F=G_{11}$, i.e., $v$ has exactly two neighbors in $V(F-C)$. Hence for any $v\in V(C)$, we have $d_{G}(v)\ge d_{G'}(v)\ge 5$ and this implies that $\{v\in V(C): d_G(v)=5\}\subseteq \{v\in V(C): d_{G'}(v)=5\}$. It follows that 
$$|\{ v\in V(C): d_G(v)=5\}|\leq |\{ v\in V(C): d_{G'}(v)=5\}|.$$
Moreover, every vertex in $V(F-C)$ has degree exactly $5$ in $G'$. 
Hence we have
\begin{align*}
N_t(C_5,G) & =|\{ v\in V(G): d_G(v)=5\}|\\
&\leq |v\in V(G'): d_{G'}(v)=5|+|\{ v\in V(\textrm{Int}_{G}(C)): d_G(v)=5\}|  -|V(\textrm{Int}_{G'}(C))|\\&
\le |v\in V(G'): d_{G'}(v)=5|+|V(\textrm{Int}_{G}(C))| -|V(\textrm{Int}_{G'}(C))|\\
&=  N_t(C_5, G')+(n-n').
\end{align*}

Next we consider non-trivial separating $5$-cycles in $G$.  It is clear that a non-trivial separating $5$-cycle of $G$ that has no vertex in $V(H-C)$ is also a  non-trivial separating $5$-cycle of $G'$ with no vertex in $V(F-C)$. For non-trivial separating $5$-cycles in $G$ using both vertices in the interior and exterior of $C$, we claim that there is no such cycle in $G$. Suppose $D$ is a non-trivial separating $5$-cycle using vertices in $\textrm{Int}_{G}(C)$ and $\textrm{Ext}_{G}(C)$. Similarly, we see that $D$ uses two non-adjacent vertices in $C$ and we may assume that $v_2, v_5\in V(D)$. Moreover, since $C$ is the unique special $5$-cycle in $H$, we know that every $v_2v_5$-path $P$ such that $V(P)\backslash\{v_2,v_5\} \subseteq V(H-C)=V(\textrm{Int}_{G}(C))$ has length at least three by Observation~\ref{obs:non-adjacent_path_length}. It follows that $D$ should use two vertices in $\textrm{Int}_{G}(C)$, say $u,v$, and exactly one vertex in $\textrm{Ext}_{G}(C)$, say $w$. Hence we may assume that $D=v_2uvv_5wv_2$. Since $G$ is $5$-connected, $w$ should be adjacent to $v_1$. Let $D_1=v_2uvv_5v_1v_2$. Since $D$ is non-trivial, observe that the interior of $D_1$ contains at least two neighbors of $v_1$ and so $D_1$ is also a non-trivial separating $5$-cycle in $G$. But we know that $|V(\overline{D_1})|<|V(\overline{C})|$, contradicting the minimality of $\overline{C}$. Hence $G$ has no non-trivial separating $5$-cycle using both vertices in the interior and exterior of $C$. Observe that both the closed interiors of $C$ in $G$ and $G'$, i.e., $H$ and $F$, contain exactly one non-trivial separating $5$-cycle of $G$, which is $C$. Thus we have that 
$$N_s(C_5,G) = N_s(C_5, G').$$ Therefore, the number of separating $5$-cycles in $G$ is  
$$N_t(C_5, G)+N_s(C_5,G)\le N_t(C_5,G')+N_s(C_5,G')+(n-n').$$
Recall that if $n'\geq 20$, it follows from induction that $G'$ has at most $(3n'-38)$ separating $5$-cycles, i.e., $N_t(C_5,G')+N_s(C_5,G')\le 3n'-38$. Hence, 
\begin{align*}
    N_t(C_5,G)+N_s(C_5,G)  
    \leq (3n'-38) + (n-n')
    =n+2n'-38 <3n-38,
\end{align*}
and this implies that $G$ has less than $(3n-38)$ separating $5$-cycles if $n'\ge 20$. Otherwise, suppose $n'\leq 19$. By checking every value of the maximum number of separating $5$-cycles in Table \ref{table:small_values}, we see that the number of separating $5$-cycles in $G$ is also less than $3n-38$. Therefore, in this case, $G$ has less than $(3n-38)$ separating $5$-cycles and hence, less than $(9n-50)$ many $5$-cycles.\\

\textbf{Case 2}: $|V(H)|=11$, i.e.,  $H=G_{11}$. Let $H_0:=H$. If there exists a vertex $u_1\in V(G)\backslash V(H_0)$ adjacent to three consecutive vertices of $C_0=C$, 
then among all such choices of $u_1$ we pick a vertex $u_1$ that maximizes the number of special $5$-cycles in $H_1:= G[V(H_0)\cup\{u_1\}]$. Let
$C_1$ denote its outer cycle (which is also a $5$-cycle). Similarly, for each $H_i$ with $i\geq 1$, if there exists a vertex $u_{i+1}\in V(G)\backslash V(H_i)$ adjacent to three consecutive vertices of $C_i$, we also pick a vertex $u_{i+1}$ such that  $H_{i+1}:= G[V(H_i)\cup \{u_{i+1}\}]$ has maximum number of special $5$-cycles among all choices of $u_{i+1}$.
Let $C_{i+1}$ denote the outer cycle of $H_{i+1}$. We repeat the above process until we cannot find such vertex. We then obtain a sequence of near triangulations $H_0\subseteq H_1\subseteq \cdots \subseteq H_t\subseteq G$ with $u_1, u_2, \cdots, u_t$ being the new vertices added in each step for $t\geq 1$. 

\begin{claim}\label{cl:cloning}
\begin{itemize}
\item [(i)] $H_t$ is a near triangulation on $11+t$ vertices with no separating triangle or $4$-cycle and its outer cycle is a $5$-cycle; 
\item [(ii)] and the number of special $5$-cycles in $H_t$ is at most $2t$ for $t\ge 1$.
\end{itemize}   
\end{claim}

\begin{proof}Clearly (i) holds since $H_t$ is a subgraph of $G$, which is $5$-connected. We will show (ii) in the remaining proof.
Observe that $H_1$ is unique by symmetry and it is not hard to check that $H_1$ has exactly two special $5$-cycles which are $C_0$ (the outer cycle of $H_0$) and $C_1$ (the outer cycle of $H_1$). Hence we may assume $t\ge 2$. Recall that for each $0\leq i\leq t$, $C_i$ denotes the outer $5$-cycle of $H_i$. Suppose $C_i=x_ia_ib_ic_i d_ix_i$ (in the clockwise order) for each $i$ with $0\le i\le t-1$ such that $u_{i+1}$ is adjacent to $x_i, a_i, d_i$ in $H_{i+1}$. Then $C_i=u_ia_{i-1}b_{i-1}c_{i-1}d_{i-1}u_i$ for each $i\in [t]$. Observe that for each $i\in [t]$ and a special $5$-cycle $D$ in $H_i$, either $D$ is a special $5$-cycle in $H_{i-1}$ or $D$ contains the vertex $u_i$. For each $i\in [t]$, let $\ell_i$ denote the number of special $5$-cycles containing $u_i$ in $H_i$. Moreover, let $\ell_0$ denote the number of special $5$-cycles in $H_0=G_{11}$. Then it follows that
$H_t$ has exactly $\sum_{i=0}^t \ell_i$ special $5$-cycles. 

\begin{figure}[htb]
\begin{center}
      \resizebox{4cm}{!}{\input{Ci.tikz}}
      \caption{$H_i$ with outer cycle $C_i=u_ia_{i-1}b_{i-1}c_{i-1}d_{i-1}u_i$ for $i\in [t]$}
      \label{fig:Ci}
\end{center} 
\end{figure}
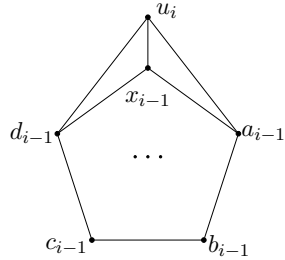
Recall that $C_i=u_ia_{i-1}b_{i-1}c_{i-1}d_{i-1}u_i$ for each $i\in [t]$. We first show that for each $i\in [t]$, every special $5$-cycle $D$ in $H_i$ such that $D$ contains the vertex $u_i$ must use the edges $u_ia_{i-1}, u_id_{i-1}$.  Suppose not and then $u_ix_{i-1}\in E(D)$. Since $D$ is special in $H_i$, observe that $D-u_i$ induces a separating $4$-cycle in $G$, giving a contradiction. Therefore, $u_ia_{i-1}, u_ib_{i-1}\in E(D)$. Now we consider the value of $\ell_i$ for each $i\in \{0, 1, \ldots, t\}$. We know that $\ell_0=1$ (since the outer cycle $C_0$ of $H_0=G_{11}$ is the only special $5$-cycle in $H_0$) and $\ell_1=1$ (since the outer cycle $C_1$ of $H_1$ is the only special $5$-cycle in $H_1$ containing $u_1$). Moreover, observe that for each $i\in [t]$, we have $\ell_i\ge 1$ as $C_i$ is a special $5$-cycle containing $u_i$ in $H_i$.  We claim that $\ell_i\le 2$ for any $i\in [2,t]$. We first show that $\ell_2 \le 2$.   
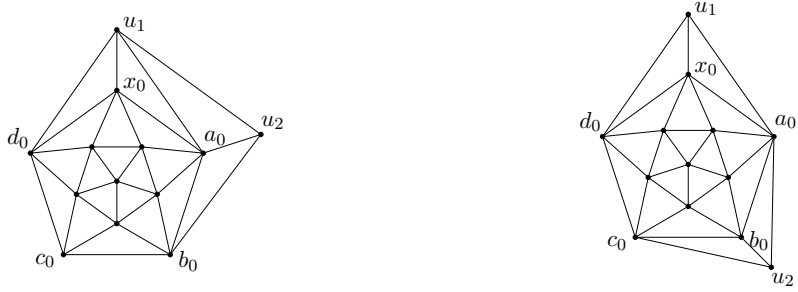
\begin{figure}[htb]
  \centering
  \begin{minipage}{.3\textwidth}
    \centering
     \resizebox{4cm}{!}{\input{H2_case1.tikz}}
  \end{minipage}
  \hspace{2cm}
  \begin{minipage}{.3\textwidth}
    \centering
   \resizebox{3.2cm}{!}{\input{H2_case2.tikz}}
  \end{minipage}
  \caption{Case $x_1=a_0$ (left) and Case $x_1=b_0$ (right)}
  \label{fig:case-x1}
\end{figure}

Let $D$ be a special $5$-cycle containing $u_2$ in $H_2$ such that $D\ne C_2$ (if $D$ exists). Note that $C_2=u_2a_1b_1c_1d_1u_2$ is the outer cycle of $H_2$, $N_{H_2}(u_2)=\{x_1, a_1, d_1\}$ and $D$ should use the edges $u_2a_1, u_2d_1$. Recall that $C_1=u_1a_0b_0c_0d_0u_1$ is the outer cycle of $H_1$ and $d_{H_1}(u_1)=3$.  Hence, we know that $x_1\in \{ a_0, b_0, c_0, d_0\}$. Assume that $x_1\in \{a_0, d_0\}$ and without loss of generality we may assume $x_1=a_0$ (See Figure~\ref{fig:case-x1}). Note that $(b_0, u_1)=(a_1,d_1)$ and thus $u_2b_0, u_2u_1$ are two edges in $D$. Similarly we know that $u_1a_0\notin E(D)$. We show that $u_1d_0\in E(D)$. Assume for contradiction that $u_1x_0\in E(C)$. Observe that $x_0a_0b_0$ is the unique $x_0b_0$-path of length two in $H_2$. Hence $D$ must be $u_2b_0a_0x_0u_1u_2$, which is impossible as $D$ is special. Thus $u_1d_0\in E(D)$. Since $D\ne C_2$, there exists a $b_0d_0$-path of length two that is different from $b_0c_0d_0$, which is impossible in $H_2$. Hence $D$ does not exist and $\ell_2=1$ if $x_1\in \{a_0, d_0\}$. We may now assume $x_1\in \{b_0, c_0\}$ and may assume that $x_1=b_0$ (See Figure~\ref{fig:case-x1}). We know that $D$ contains the edges $u_2a_0, u_2c_0$. It is not hard to check (with similar analysis as before) that $D$ can only be $u_2c_0d_0x_0a_0u_2$. This implies that $\ell_2=2$ if $x_1\in \{b_0,c_0\}$. Therefore, $\ell_2\le 2$.

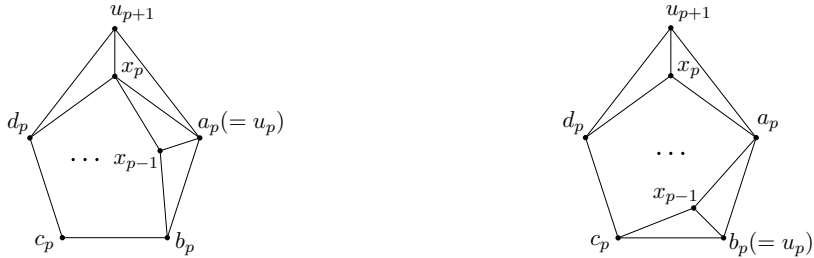
\begin{figure}[htb]
  \centering
  \begin{minipage}{.3\textwidth}
    \centering
     \resizebox{4cm}{!}{\input{Hp_case1.tikz}}
  \end{minipage}
  \hspace{2cm}
  \begin{minipage}{.3\textwidth}
    \centering
   \resizebox{3.65cm}{!}{\input{Hp_case2.tikz}}
  \end{minipage}
  \caption{Case $u_p=a_p$ (left) and Case $u_p=b_p$ (right)}
  \label{fig:case-up}
\end{figure}

 Suppose $\ell_i\ge 3$ for some $i\in [2,t]$. Let $p$ be the minimum integer such that $\ell_{p+1}\ge 3$. Since $\ell_0=1, \ell_1=1$ and $\ell_2\le 2$, we know that $p\ge 2$. Note that $u_p\in V(C_p)=\{x_p, a_p,b_p, c_p, d_p\}$ and $u_p\ne x_p$ by the process. Suppose $u_p=a_p$ (See Figure~\ref{fig:case-up}). Similar to before, every special $5$-cycle in $H_{p+1}$ containing $u_{p+1}$ must contain $u_{p+1}a_p, u_{p+1}d_p, a_pb_p$ as edges. Therefore, $\ell_{p+1}\ge 3$ implies that there are at least three $b_pd_p$-paths of length two in $G$ which give a separating $4$-cycle in $G$, a contradiction. Thus $u_p\ne a_p$. Similarly, we have that $u_p\ne d_p$ and $u_p\in \{b_p, c_p\}$. By symmetry, we may assume $u_p=b_p$ (See Figure~\ref{fig:case-up}). This implies that $C_{p-1}=x_{p-1}c_pd_px_pa_px_{p-1}$ and  we have that $u_{p-1}\in V(C_{p-1})\backslash \{ x_{p-1}, x_p\}=\{a_p, c_p, d_p\}$. Suppose $u_{p-1}=d_p$. By a similar argument, each special $5$-cycle containing $u_{p+1}$ in $H_{p+1}$ contains the edges $u_{p+1}a_p, u_{p+1}d_p, c_pd_p$ and there are at least three $a_pc_p$-paths of length two in $H_{p+1}$ as $\ell_{p+1}\ge 3$. This contradicts that $G$ has no separating $4$-cycle. Thus $u_{p-1}\in \{a_p, c_p\}$. Note that $C_{p+1}$ is the only special $5$-cycle containing $u_{p+1}$ and $u_{p}$, and so the other $\ell_{p+1}-1$ many special $5$-cycles in $H_{p+1}$ containing $u_{p+1}$ are also special $5$-cycles in $G[V(H_{p-1})\cup \{u_{p+1}\}]$. By the choice of $u_{p}$, we know $\ell_p\ge \ell_{p+1}-1\ge 2$ and thus $\ell_p=2$, i.e., $H_p=G[V(H_{p-1})\cup \{u_p\}]$ has exactly two special $5$-cycles containing $u_p$. If $u_{p-1}=c_p$, then for any special $5$-cycle $D$ containing $u_p$ in $H_p$ we have $a_pu_pc_pd_p\subseteq D$. Since $\ell_p=2$, there exist two $a_pd_p$-paths of length two in $H_p$. These two paths together with the path $a_pu_{p+1}d_p$ give a separating $4$-cycle, a contradiction. Thus, $u_{p-1}=a_p$. Similarly, by the fact that $G[V(H_{p-1)}\cup \{u_{p+1}\}]$ has two special $5$-cycles containing $u_{p+1}$, we know that there are two $x_{p-1}d_p$-paths of length two in $H_{p-1}$ (one is $x_{p-1}c_pd_p$ and denote the other by $x_{p-1}v d_p$). It follows that $c_p$ is adjacent to $v$ as $G$ has no separating $4$-cycle. Observe that now $C_p=x_pa_pu_pc_pd_px_p$ is the only special $5$-cycle containing $u_p$ in $H_p$, contradicting that $\ell_{p}=2$. Therefore, we have that $\ell_i\le 2$ for $i\in [2,t]$, and so
 $$\sum_{i=0}^t \ell_i=\ell_0+\ell_1+\sum_{i=2}^t \ell_i=2+ \sum_{i=2}^t \ell_i\le 2+2(t-1)=2t.$$
Thus $H_t$ has at most $2t$ special $5$-cycles for $t\ge 1$.
\end{proof}

Note that there exists no vertex $u\in V(G)\backslash V(H_t)$ such that $u$ is adjacent to three consecutive vertices in $C_t$. We show that for any two non-adjacent vertices $w_1,w_2$ in $C_t$, $w_1$ and $w_2$ cannot have a common neighbor in the exterior of $C_t$. Suppose there exists $u\in V(G)\backslash V(H_t)$ such that $w_1u, w_2u\in E(G)$. Denote the path between $w_1$ and $w_2$ in $C_t$ that has length two by $P$ and assume that $P=w_1yw_2$. Then $\{u, w_1, y, w_2\}$ induces a $4$-cycle $uw_1yw_2u$ in $G$. Since $G$ has no separating $4$-cycle, either $uy$ is an edge of $G$ or $w_1w_2$ is an edge of $G$. Observe that if $w_1w_2\in E(G)$ then $C_t-y+w_1w_2$ is a separating $4$-cycle in $G$, a contradiction. It follows that $u$ is adjacent to $y$ and this implies that there exists $u\in V(G)\backslash V(H_t)$ such that $u$ is adjacent to three consecutive vertices in $C_t$, a contradiction. 

First, we claim that $V(H_t)\neq V(G)$. Suppose for contradiction that $V(H_t)=V(G)$. Then $n=|V(G)|=|V(H_t)|=11+t$ and it follows that $t=n-11\ge 10$. We know that $C_t=u_{t}a_{t-1}b_{t-1}c_{t-1}d_{t-1}u_{t}$ and $C_{t-1}=x_{t-1}a_{t-1}b_{t-1}c_{t-1}d_{t-1}x_{t-1}$ such that $N_{H_t}(u_t)=\{d_{t-1}, x_{t-1}, a_{t-1}\}$.  Since $G$ is $5$-connected and $d_{H_t}(u_t)=3$, $u_t$ should be adjacent to $b_{t-1},c_{t-1}$ and then $G=H_t+\{ u_tb, u_tc\}$. Observe that $u_{t-1}\in \{ a_{t-1}, b_{t-1}, c_{t-1}, d_{t-1}\}$. By the fact that $d_{H_{t-1}}(u_{t-1})=3$, it follows that $d_{H_t}(u_{t-1})=4$ if $u_{t-1}\in \{ a_{t-1}, d_{t-1}\}$; and  $d_{H_t}(u_{t-1})=3$ if  $u_{t-1}\in \{ b_{t-1}, c_{t-1}\}$. Thus $u_{t-1}$ has degree $4$ in $G$ as $G=H_t+\{ u_tb_{t-1}, u_tc_{t-1}\}$, contradicting that $G$ is $5$-connected. Therefore, $V(H_t)\neq V(G)$.

Let $G'$ be obtained from $G$ by contracting the interior of $C_t$, i.e., $H_t-C_t$, into a single vertex $v'$. We claim that $G'$ is $5$-connected. Suppose not. Then $G'$ has a separating triangle or $4$-cycle, say $D$. Since $G$ is $5$-connected, $D$ must use the vertex $v'$ and one vertex $u$ in the exterior of $C_t$. This implies that $D$ should be a $4$-cycle and $D$ contains two vertices $w_1, w_2$ in $C_t$. Note that $w_1,w_2$ are not adjacent in $C_t$, otherwise $\{u, w_1, w_2\}$ induces a separating triangle in $G$. Hence $w_1, w_2$ are not adjacent in $C$. Then $u$ is a vertex in the exterior of $C_t$ such that $u$ is a common neighbor of two non-adjacent vertices in $C_t$, a contradiction. Therefore, $G'$ is $5$-connected and we know that $n'=|V(G')|=|V(G)|-(|V(H_t)|-5)+1=n-(6+t)+1=n-5-t$. Hence $t=n-5-n'$. Observe that $d_{G}(v)\ge d_{G'}(v)\ge 5$ for every $v\in V(C_t)$ since $v$ has at least one neighbor contained in $V(H_t-C_t)$ by (i) of Observation~\ref{obs:int_degree}. Similar to Case $1$, we can show that the number of trivial separating $5$-cycles in $G$ satisfies that
$$N_t(C_5, G)\le N_t(C_5, G')+(n-n').$$ 

Now we consider non-trivial separating $5$-cycles in $G$. Let $D$ be a non-trivial separating $5$-cycle of $G$. Note that $C_t$ is a non-trivial separating $5$-cycle in $G$, a trivial separating $5$-cycle in $G'$, and a special $5$-cycle in $H_t$. If $D$ is contained in $H_t$, it follows that $D$ is 
a special $5$-cycle of $H_t$. By Claim~\ref{cl:cloning}, $H_t$ has at most $2t$ special $5$-cycles if $t\ge 1$. When $t=0$, we know that $H_0=G_{11}$ has exactly one special $5$-cycle which is its outer cycle. We may now assume that $D\nsubseteq H_t$. Then $D\ne C_t$. Suppose $D$ contains no vertex in the interior of $C_t$, i.e., $V(H_t-C_t)$. It follows that $D$ is also a non-trivial separating $5$-cycle of $G'$ contained in $G'-v'$. Assume that $D$ uses vertices both in the interior and the exterior of $C_t$. Then $D$ must use two vertices $w_1, w_2$ in $C_t$. It follows from (i) of Observation~\ref{obs:int_degree} that $D$ has no chord and so $w_1, w_2$ are not adjacent in $C_t$. This implies that $D$ should contain at least two vertices in $\textrm{Ext}_G(C_t)$, otherwise $D$ uses a unique vertex $u\in V(\textrm{Ext}_G(C_t))=V(G)\backslash V(H_t)$ and $w_1, w_2$ are adjacent to $u$, giving a contradiction. Hence $D$ uses two non-adjacent vertices of $C_t$, two vertices in $\textrm{Ext}_G(C_t)$ and one vertex $w_D$ in $\textrm{Int}_G(C_t)$.  Let $P=w_1yw_2$ be the $w_1w_2$-path of length two in $C_t$. Then $w_Dw_1yw_2w_D$ is a $4$-cycle in $H_t$, and $w_D$ is adjacent to $y$ since $H_t$ has no separating triangle or $4$-cycle. Let $D':=D-w_D+\{w_1 v', w_2 v'\}=G'[V(D-w_D)\cup\{v'\}]$. We show that $D'$ is a non-trivial separating $5$-cycle of $G'$. If $y$ is in the interior of $D$, then $\textrm{Int}_{G'}(D')=\textrm{Int}_G(D)$ as $w_Dy\in E(G)$ and $\textrm{Ext}_{G'}(D)$ contains the two vertices in $V(C_t)\backslash V(P)$. Since $D$ is non-trivial in $G$, $\textrm{Int}_G(D)=\textrm{Int}_{G'}(D')$ has at least two vertices. It follows that $D'$ is a non-trivial separating $5$-cycle in $G'$. Similarly, we can show $D'$ is also non-trivial in $G'$ if $y$ is in the exterior of $D$. Let $D_1, D_2$ be two distinct non-trivial separating $5$-cycles of $G$ using vertices in $\textrm{Int}_{G}(C_t)$ and $\textrm{Ext}_{G}(C_t)$, and let $w_{D_i}$ be the unique vertex contained in $\textrm{Int}_{G}(C_t)$ of $D_i$. Let $D_i':=G'[V(D_i-w_{D_i})\cup \{v'\}]$ for each $i\in [2]$. We claim that $D_1'\ne D_2'$. Suppose $D_1'=D_2'$. This implies that $D_1-w_{D_1}=D_2-w_{D_2}$ in $G$. Since $D_1\ne D_2$, we know that $w_{D_1}\ne w_{D_2}$ which is impossible as $G$ is $5$-connected. Thus $D_1'\ne D_2'$. Therefore, the number of non-trivial separating $5$-cycles of $G$ satisfies that 
$N_s(C_5, G)\le 1+ N_s(C_5,G')$ if $t=0$; and $N_s(C_5, G)\le 2t+ N_s(C_5,G')$ if $t\ge 1$. Recall that $n'=n-5-t$ and $t=n-5-n'$. It follows that if $t=0$, then $n'=n-5\ge 16$ and
\begin{align*}
N_t(C_5,G)+N_s(C_5,G)&\le N_t(C_5,G')+N_s(C_5,G')+(n-n')+1\\
&= N_t(C_5,G')+N_s(C_5,G')+n-(n-5)+1\\
&= N_t(C_5,G')+N_s(C_5,G')+6;
\end{align*}
and if $t\ge 1$,
\begin{align*}
N_t(C_5,G)+N_s(C_5,G) &\le N_t(C_5,G')+N_s(C_5,G')+(n-n')+2t\\
&=N_t(C_5,G')+N_s(C_5,G')+(n-n')+2(n-5-n')\\
&=N_t(C_5,G')+N_s(C_5,G')+3n-3n'-10.
\end{align*}

We show that if $n'\geq 20$ then the number of separating $5$-cycles in $G$ is less than $3n-38$. Suppose $n'\geq 20$. By induction, we have $N_t(C_5,G')+N_s(C_5,G')\leq 3n'-38$. Hence when $t=0$, $n'=n-5$ and 
$$N_t(C_5,G)+N_s(C_5,G) \leq (3n'-38) + 6=3(n-5)-38+6=3n-47<3n-38;$$ and when $t\ge 1$,
$$N_t(C_5,G)+N_s(C_5,G) \leq (3n'-38) + 3n-3n'-10= 3n-48<3n-38.$$ 
On the other hand, if $14\leq n' \leq 19$, by checking every value of the maximum separating $5$-cycles of $n'$-vertex $5$-connected planar triangulations in Table \ref{table:small_values} we also see that $G$ has smaller than $3n-38$ separating $5$-cycles. Therefore, if $n'\ge 14$ then $G$ has smaller than $(9n-50)$ many $5$-cycles. We may now assume $n'\le 13$.

Since there are no $5$-connected planar triangulations on $13$ vertices, the only remaining case is when $n'=12$. If $n'=12$, note that $G'-v'$ has exactly $11$ vertices and it is a near triangulation with its outer cycle being a $5$-cycle and no separating triangle or $4$-cycle. It follows that  $G-(H_t-C_t)=G'-v'=G_{11}$ and every vertex in $C_t$ has exactly two neighbors in the exterior of $C_t$. Note that $G'$ has no non-trivial separating $5$-cycle, i.e., $N_s(C_5,G')=0$, as $G'-v'=G_{11}$.
We show that in this case, $G$ has at most $3n-38$ separating $5$-cycles and the equality holds if and only if $G$ is the graph $D_1$ (when $n$ is even) and $D_2$ (when $n$ is odd), see Figure \ref{fig:extremal}. 

 We first claim that (i) there are at least three vertices in $C_t$ that have degree at least four in $H_t$; (ii) moreover, if $t\ge 3$, there exists at least one vertex in $H_t-C_t$ such that its degree is at least six in $G$. We will prove (i) by induction on $t$. Clearly (i) is true for $H_0=G_{11}$, as every vertex in $C_0$ has degree $4$ in $H_0$. Suppose (i) is true for $H_j$ for some $j\geq 0$. Recall that we use $C_i:= x_i a_i b_i c_i d_i x_i$ to denote the outer cycle of $H_i$ for each $i$ with $0\leq i\leq t-1$ such that $N_{H_{i+1}}(u_{i+1})=\{x_i, a_i, d_i\}$. We aim to show that (i) holds for $H_{j+1}$. It is clear that $d_{H_j}(x)\geq 3$ for every $x\in V(C_j)$. Hence $d_{H_{j+1}}(d_j)\geq 4$ and $d_{H_{j+1}}(a_j) \geq 4$. Thus, if $d_{H_{j+1}}(c_j) = d_{H_j}(c_j)\geq 4$ or $d_{H_{j+1}}(b_j) = d_{H_j}(b_j)\geq 4$, then we are done. Otherwise, $d_{H_j}(c_j) = d_{H_j}(b_j)=3$. However, this implies that there exists a vertex $v\in V(\textrm{Int}_G(C_j))$ such that $v$ is adjacent to $a_j, b_j, c_j, d_j$ and the $4$-cycle $u_{j+1}a_jvd_ju_{j+1}$ is separating in $G$, a contradiction. Thus we can conclude by induction that there are at least three vertices in $C_t$ that have degree at least four in $H_t$. For (ii), it suffices to show that there exists at least one vertex in $V(H_3-C_3)$ with degree at least $6$ in $H_3$. Suppose not and then all vertices in $V(H_3-C_3)$ have degree $5$ in $H_3$. This also implies that every vertex in $V(H_2-C_2)$ have degree $5$ in $H_2$. We know that $H_1$ is unique and $d_{H_1}(a_0)=5$ and $d_{H_1}(d_0)=5$. Note that $x_1\in \{a_0, b_0, c_0, d_0\}$ and $x_1\in V(H_2-C_2)$ such that $d_{H_2}(x_1)=d_{H_1}(x_1)+1$. Since $x_1$ has degree $5$ in $H_2$, we obtain that $x_1\in \{b_0, c_0\}$ and may assume $x_1=b_0$. Observe that $c_0, d_0$ have degree $5$ and $a_0$ has degree $6$ in $H_2$. Since $x_2\in \{c_0, d_0, a_0\}$ and $d_{H_3}(x_2)=d_{H_2}(x_2)+1$, $x_2$ must have degree at least $6$ in $H_3$ (with $x_2\in V(H_3-C_3)$) which is a contradiction.

 Recall that $n'=n-5-t$. Since $n'=12$, $t=n-5-n'=n-17\ge 4$. This implies that there are at least four vertices of degree at least six in $G$ (at least three in $V(C_t)$ and at least one in $V(H_t-C_t)$), and so $N_t(C_5,G)\le n-4$. Since $t\ge 4$, $H_t$ has at most $2t$ special $5$-cycles. Thus $N_s(C_5,G)\le 2t+N_s(C_5,G')=2t+0=2t=2(n-5-n')=2(n-17)=2n-34$. Therefore, 
 $$N_t(C_5, G)+N_S(C_5, G)\le (n-4)+(2n-34)=3n-38.$$
 This implies that $G$ has at most $3n-38$ separating $5$-cycles, and the equality holds if and only if exactly three vertices in $C_t$ have degree at least four in $H_t$ (i.e., exactly two vertices in $C_t$ have degree three in $H_t$), exactly one vertex in $H_t-C_t$ having degree at least six, and $H_t$ has exactly $2t$ special $5$-cycles (i.e., $\ell_i=2$ for every $2\leq i\leq t$). We claim that $H_t$ is unique under those assumptions. 
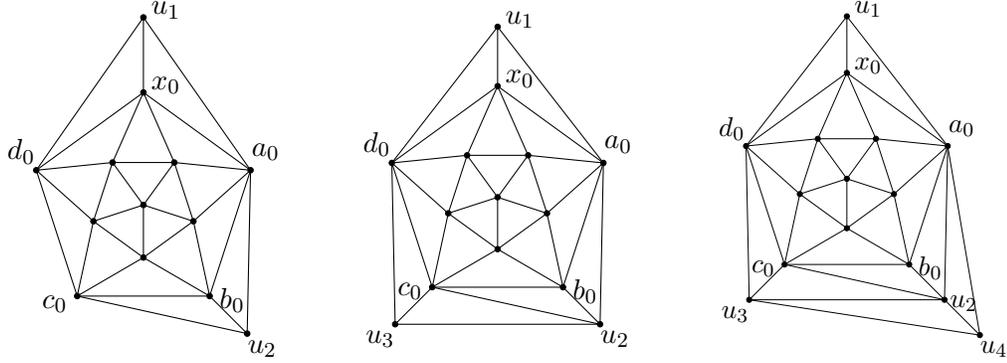
\begin{figure}[htb]
  \centering
  \begin{minipage}{.15\textwidth}
    \centering
     \resizebox{4cm}{!}{\input{H2_case2.tikz}}
  \end{minipage}
  \hspace{2cm}
  \begin{minipage}{.15\textwidth}
    \centering
   \resizebox{3.95cm}{!}{\input{H3.tikz}}
  \end{minipage}
  \hspace{2cm}
  \begin{minipage}{.15\textwidth}
    \centering
   \resizebox{4.2cm}{!}{\input{H4.tikz}}
\end{minipage}
  \caption{Graphs $H_2$ (left), $H_3$ (middle), $H_4$ (right).}
  \label{fig:unique_Ht}
\end{figure}
 We know $H_1$ is unique as $H_0$ is symmetric. Since $\ell_2=2$, we have that $x_1\in \{b_0, c_0\}$ and then $H_2$ is also unique by symmetry. Thus we may assume $x_1=b_0$. Next we show $H_3$ is unique. Observe that $x_2\in \{a_0, c_0, d_0\}$. If $x_2=a_0$, then $\ell_3=1$. Hence $x_2\in \{c_0, d_0\}$ and $H_3$ is unique by symmetry. We may assume $x_2=c_0$ and then $c_0$ has degree $6$ in $H_t$. Since $H_t$ has exactly one vertex of degree $6$ contained in $V(H_t-C_t)$, we then obtain that $x_3=u_2$ and moreover $x_i=u_{i-1}$ for each $i\in \{3, \ldots, t-1\}$ if $t\ge 4$. This implies that $H_t$ is unique. Since $G-(H_t-C_t)=G_{11}$ is symmetric, it follows that $D_1$ (when $n$ is even) or $D_2$ (when $n$ is odd) is the unique extremal graph satisfying the constraints above.
 This completes the proof of Theorem \ref{thm:5-cycle-5-connected}.
 \end{proof}

\section{Number of $k$-cycles in $5$-connected planar triangulations}\label{sec:k-cycle}

In this section, we prove Theorem \ref{thm:stronger}, which implies Theorem \ref{thm:long-cycles}. We use the following lemma about planar graphs which was proved (but never published) by Barnette~\cite{Barnette} and independently by Borodin~\cite{Borodin1989} (see \cite{JV2013} for more history).

\begin{lemma}\label{lem:light-edges}\cite{Barnette, Borodin1989}
Let $G$ be a planar graph with minimum degree at least $3$. Then $G$ contains two adjacent vertices $u$ and $v$ such that $d(u)+d(v)\leq 13$.
\end{lemma}

Based on Lemma \ref{lem:light-edges}, we show the following lemma.

\begin{lemma}\label{lem:light-edges-min-degree-2}
Let $G$ be a planar graph with minimum degree at least $2$ such that every two distinct vertices have at most two common neighbors. Then $G$ contains two adjacent vertices $u$ and $v$ such that $d(u)+d(v)\leq 39$.
\end{lemma}
\begin{proof}
We may assume that $G$ has minimum degree exactly $2$, otherwise the lemma holds by Lemma \ref{lem:light-edges}. Moreover, we may assume that every neighbor of a degree two vertex in $G$ has degree at least $38$, as otherwise the lemma holds trivially. 

Let $G'$ be the simple graph obtained from $G$ by suppressing all the degree-$2$ vertices of $G$ and removing the multi-edges. 
Observe that since every two distinct vertices in $G$ have at most two common neighbors, and there are no two adjacent degree-$2$ vertices in $G$, we obtain that every two vertices in $G$ have at most two paths between them with all the internal vertices being degree two. Hence after suppressing all the degree-$2$ vertices in $G$, every two vertices have at most three multi-edges between them (possibly including the edge between the two vertices if they are adjacent).
It follows that $d_{G'}(x)\geq d_{G}(x)/3$, i.e., $d_G(x)\leq 3 d_{G'}(x)$ for every $x\in V(G')$. Moreover, observe that if $x\in V(G')$ has no neighbor of degree two in $G$ then $d_{G'}(x)=d_G(x)$. This implies that if $x\in V(G')$ satisfies that $d_{G'}(x)<d_G(x)$, then $x$ has a neighbor of degree two in $G$.

Suppose first that $\delta(G')\leq 2$ and let $u$ be a vertex in $G'$ with degree at most $2$. Then $d_{G'}(u)\leq 2<3\leq d_G(u)$ and hence there exists a vertex $w\in N_G(u)$ such that $d_G(w)=2$. This implies that $u$ should have degree at least $38$ in $G$. However, $d_{G}(u)\le 3d_{G'}(u)\le 6$, giving a contradiction. Hence we can assume $\delta(G')\geq 3$. Now by Lemma \ref{lem:light-edges}, there exist $uv\in E(G')$ such that $d_{G'}(u)+d_{G'}(v)\leq 13$. Thus we have
$$d_G(u)+d_G(v)\leq 3(d_{G'}(u)+d_{G'}(v))\leq 39.$$
If $uv\in V(G)$, then we are done. Otherwise, there exists a degree two vertex $w\in V(G)$ such that $ N_G(w)=\{u,v\}$. Since $d_G(u)+d_G(v)\leq 39$, we have that $\min\{d_G(u),d_G(v)\}\leq 19$. Without loss of generality, suppose $d_G(u)\leq 19$. Then it follows that 
$$d_G(u)+d_G(w)\leq 19+2=21,$$
and we are done.
\end{proof}

Now we are ready to prove Theorem \ref{thm:stronger}, which we restate here for convenience.

\stronger*

\begin{proof}

Let $G$ be an $m$-edge planar graph such that every two distinct vertices have at most two common neighbors. Given an integer $k$ with $k\ge 6$, we need to show that $G$ has at most $ C \cdot m^{\floor{k/3}}$ many $k$-cycles, where $C:=342 \cdot \frac{4^{\floor{k/3}}}{\floor{k/3}!}$.

Let $p:=\floor{k/3}$. We first consider the case when $k\equiv 0 \pmod{3}$.  Then $k=3p$ and $p\geq 2$. We  enumerate all the $k$-cycles $C$ in $G$. Without loss of generality, label the vertices of $C$ by $C=v_0 v_1 v_2\cdots v_{3p-1} v_0$. Consider the possible choices of edges in $G$ for the edges of $C$ in $\{v_0v_1, v_3v_4,\cdots, v_{3p-3}v_{3p-2}\}$. There are at most $\binom{m}{p} \cdot 2 ^p \leq \frac{(2m)^p}{p!}$ choices for those edges in $C$. Note that for a fixed set of edges $\{v_0v_1, v_3v_4,\cdots, v_{3p-3}v_{3p-2}\}$, there are at most two choices for each of the vertices in $\{v_2, v_5, \cdots, v_{3p-1}\}$, since every two distinct vertices in $G$ have at most two common neighbors. It follows that the number of $k$-cycles in $G$ is at most 
$$\frac{(2m)^p}{p!} \cdot 2^p = \frac{4^p}{p!} m^{p}\leq C m^{\floor{k/3}}.$$

Now we show Theorem \ref{thm:stronger} when $k\equiv 2 \pmod 3$ or $k\equiv 1 \pmod 3$. We only show the case when $k\equiv 2 \pmod 3$ here, as the reasoning for the case $k\equiv 1 \pmod 3$ is near identical. Then $k=3p+2$ and $p\geq 2$. We apply induction on the number of edges in $G$. The cases when $m< k$ are clearly true. Thus we can assume that $m\geq k = 3p+2$. Suppose first that $G$ has a vertex $u$ of degree $1$, and suppose that $v$ is the unique neighbor of $u$. Note that the edge $uv$ does not lie in any cycle. Hence by induction,
$$N(C_k,G) = N(C_k, G-uv)\leq C(m-1)^{\floor{k/3}} \leq Cm^{\floor{k/3}}.$$
Hence we may assume that $G$ has minimum degree at least $2$. By Lemma \ref{lem:light-edges-min-degree-2}, there exists an edge $uv\in E(G)$ such that $d(u)+d(v)\leq 39$. We will enumerate all the $k$-cycles $C$ in $G$. Similar to before, label the vertices of $C$ by $C=v_0 v_1 v_2\cdots v_{3p+1}v_0$.
Let $S_1$ and $S_2$ denote the set of $k$-cycles in $G$ containing and not containing the edge $uv$, respectively. We first bound $|S_2|$. Observe that $|S_2| = N(C_k, G-uv)$. It follows from induction that
$$|S_2|=N(C_k,G-uv)\leq C(m-1)^{\floor{k/3}} = C(m-1)^p.$$
Now we bound $|S_1|$. Without loss of generality, suppose that $v_0v_1$ is the edge $uv$. Since
$d(u)+d(v)\leq 39$, the number of choices for $(v_2, v_{3p+1})$ is at most 
\begin{align*}
    (d(u)-1)(d(v)-1) & =d(u)d(v)-(d(u)+d(v))+1 \\ &\le \frac{(d(u)+d(v))^2}{4}-(d(u)+d(v))+1 \\ &\le \frac{39^2}{4}-39+1=342.25.
\end{align*}
Next, similar to before, consider the possible choices of edges in $G$   for the positions (edges) $$v_4v_5, v_7v_8,\cdots, v_{3p-2}v_{3p-1}$$ in $C$. There are at most $\binom{m}{p-1} \cdot 2^{p-1}\leq \frac{(2m)^{p-1}}{(p-1)!}$ choices for those edges. Note that for fixed vertices $v_0, v_1, v_2, v_{3p+1}$ and a fixed set of edges $\{v_4v_5, v_7v_8,\cdots, v_{3p-2}v_{3p-1}\}$, there are at most two choices for each of the vertices in $\{v_3, v_6, \cdots, v_{3p}\}$, as every two distinct vertices in $G$ have at most two common neighbors. Thus, it follows that 
$$|S_1|\leq 342\cdot \frac{(2m)^{p-1}}{(p-1)!} \cdot 2^p\leq 171p \cdot \frac{4^p}{p!} m^{p-1}.$$
Therefore, we have that 
\begin{align*}
    N(C_k,G) & = |S_1|+|S_2| \\
             &\leq C(m-1)^p + 171p \cdot \frac{4^p}{p!} m^{p-1}\\
             &\leq Cm^p,
\end{align*}
as $C =342\cdot  \frac{4^p}{p!}$, and the last inequality follows from a routine calculation using the Binomial Theorem. This completes the proof of Theorem \ref{thm:stronger}.
\end{proof}

\bigskip

\noindent\textbf{Acknowledgment.} We would like to thank \'Eva Czabarka and L\'aszl\'o Sz\'ekely for early discussions on $5$-connected planar graphs with large number of Hamiltonian cycles.

\end{document}

%% file: 5_connected_extremal.tikz
\begin{tikzpicture}[scale=1, wvertex/.style={circle, draw=red, fill=red, scale=0.3}, bvertex/.style={circle, draw=black, fill=black, scale=0.3},rvertex/.style={circle, draw=red, fill=red, scale=0.3}, sbvertex/.style={circle, draw=black, fill=black, scale=0.1}]
    
    \node [bvertex] (a) at (-0.5,0) {};
    \node [bvertex](b) at (0.5,0) {};
    \node [bvertex](c) at (0,1.5) {};
    \node [bvertex](d) at (0,-1.5) {};
    
    \node [bvertex](v1) at (-1,0.5) {};
    \node [bvertex](v2) at (-1.5,0.5) {};
    \node [bvertex](v3) at (-2,0.5) {};
    \node [bvertex](v4) at (-2.5,0.5) {};
    \node [bvertex](v5) at (-3,0.5) {};
    \node [bvertex](v6) at (-3.5,0.5) {};
    \node [bvertex, label={[font=\normalsize] left:$v$}](v7) at (-4.0,0.5) {};
        
    \node [bvertex](u1) at (-1,-0.5) {};
    \node [bvertex](u2) at (-1.5,-0.5) {};
    \node [bvertex](u3) at (-2,-0.5) {};
    \node [bvertex](u4) at (-2.5,-0.5) {};
    \node [bvertex](u5) at (-3,-0.5) {};
    \node [bvertex](u6) at (-3.5,-0.5) {};
    \node [bvertex](u7) at (-4.0,-0.5) {};

    \node [sbvertex] at (-3.15,0){};
    \node [sbvertex] at (-3.25,0){};
    \node [sbvertex] at (-3.35,0){};
    \node [sbvertex] at (3.15,0){};
    \node [sbvertex] at (3.25,0){};
    \node [sbvertex] at (3.35,0){};

    \node [bvertex](w1) at (1,0.5) {};
    \node [bvertex](w2) at (1.5,0.5) {};
    \node [bvertex](w3) at (2,0.5) {};
    \node [bvertex](w4) at (2.5,0.5) {};
    \node [bvertex](w5) at (3,0.5) {};
    \node [bvertex](w6) at (3.5,0.5) {};
    \node [bvertex](w7) at (4.0,0.5) {};
    
    \node [bvertex](x1) at (1,-0.5) {};
    \node [bvertex](x2) at (1.5,-0.5) {};
    \node [bvertex](x3) at (2,-0.5) {};
    \node [bvertex](x4) at (2.5,-0.5) {};
    \node [bvertex](x5) at (3,-0.5) {};
    \node [bvertex](x6) at (3.5,-0.5) {};
    \node [bvertex, label={[font=\normalsize] right:$v'$}](x7) at (4.0,-0.5) {};

    \node [bvertex] (y1) at (-1.25,0) {};
    \node [bvertex] (y2) at (-1.75,0) {};

    \node [bvertex] (z1) at (1.25,0) {};
    \node [bvertex] (z2) at (1.75,0) {};

    \draw[-] (v7) .. controls (-1,2.4) and (1,2.4) .. (w7);
    
    \draw[-] (u7) .. controls (-1,-2.4) and (1,-2.4) .. (x7);
    
    \draw (a) -- (b);
    \draw (a) -- (v1) -- (v2) -- (v3) -- (v4) -- (v5);
    \draw (a) -- (u1) -- (u2) -- (u3) -- (u4) -- (u5);

    \draw (b) -- (w1) -- (w2) -- (w3) -- (w4) -- (w5);
     \draw (b) -- (x1) -- (x2) -- (x3) -- (x4) -- (x5);
     \draw (w1) -- (x1) -- (z1) -- (x2) -- (z2) -- (x3) -- (w3);
     \draw (x3) -- (w4) -- (x4) -- (w5) -- (x5);
       \draw (v1) -- (u1) -- (y1) -- (u2) -- (y2) -- (u3) -- (v3) -- (u4) -- (v4) -- (u5) -- (v5);
       \draw (w1) -- (z1) -- (w2) -- (z2) -- (w3);
       \draw (v1) -- (y1) -- (v2) -- (y2) -- (v3);
       \draw (y1) -- (y2);
       \draw (z1) -- (z2);
       \draw (c)--(a);  \draw (c)--(b);  \draw (c)--(v1);  \draw (c)--(v2); \draw (c)--(v3); \draw (c)--(v4);  \draw (c)--(v5);\draw (c)--(v6);  \draw (c)--(v7);
       \draw (c)--(w1);  \draw (c)--(w2); \draw (c)--(w3); \draw (c)--(w4);  \draw (c)--(w5); \draw (c)--(w6);  \draw (c)--(w7);

     \draw (d)--(a);  \draw (d)--(b);  \draw (d)--(u1);  \draw (d)--(u2); \draw (d)--(u3); \draw (d)--(u4);  \draw (d)--(u5); \draw (d)--(u6);  \draw (d)--(u7);
     \draw (d)--(x1);  \draw (d)--(x2); \draw (d)--(x3); \draw (d)--(x4);  \draw (d)--(x5); \draw (d)--(x6);  \draw (d)--(x7);

     \draw (v5)--(v6) -- (v7) -- (u7) -- (u6) -- (v6) -- (u7); \draw (u6) --(u5);

     \draw (w5)-- (w6) -- (w7) -- (x7) -- (x6) -- (w6);
     \draw (w7) -- (x6) -- (x5);

     \node (I1) at (-4.5,-0.5) {};
     \node (I2) at (4.5,0.5){};
     \draw (v7) -- (I1);\draw (I2) -- (x7);
        
\end{tikzpicture}

%% file: 5_connected_extremal2.tikz
\begin{tikzpicture}[scale=1, wvertex/.style={circle, draw=red, fill=red, scale=0.3}, bvertex/.style={circle, draw=black, fill=black, scale=0.3},rvertex/.style={circle, draw=red, fill=red, scale=0.3}, sbvertex/.style={circle, draw=black, fill=black, scale=0.1}]
    
    \node [bvertex] (a) at (-0.5,0) {};
    \node [bvertex](b) at (0.5,0) {};
    \node [bvertex](c) at (0,1.5) {};
    \node [bvertex](d) at (0,-1.5) {};
    
    \node [bvertex](v1) at (-1,0.5) {};
    \node [bvertex](v2) at (-1.5,0.5) {};
    \node [bvertex](v3) at (-2,0.5) {};
    \node [bvertex](v4) at (-2.5,0.5) {};
    \node [bvertex](v5) at (-3,0.5) {};
    \node [bvertex](v6) at (-3.5,0.5) {};
    \node [bvertex, label={[font=\normalsize] left:$v$}](v7) at (-4.0,0.5) {};
        
    \node [bvertex](u1) at (-1,-0.5) {};
    \node [bvertex](u2) at (-1.5,-0.5) {};
    \node [bvertex](u3) at (-2,-0.5) {};
    \node [bvertex](u4) at (-2.5,-0.5) {};
    \node [bvertex](u5) at (-3,-0.5) {};
    \node [bvertex](u6) at (-3.5,-0.5) {};
    \node [bvertex](u7) at (-4.0,-0.5) {};

    \node [sbvertex] at (-3.15,0){};
    \node [sbvertex] at (-3.25,0){};
    \node [sbvertex] at (-3.35,0){};
    \node [sbvertex] at (3.15,0){};
    \node [sbvertex] at (3.25,0){};
    \node [sbvertex] at (3.35,0){};

    \node [bvertex](w1) at (1,0.5) {};
    \node [bvertex](w2) at (1.5,0.5) {};
    \node [bvertex](w4) at (2,0.5) {};
    \node [bvertex](w5) at (2.5,0.5) {};
    \node [bvertex](w6) at (3,0.5) {};
    \node [bvertex](w7) at (3.5,0.5) {};
    
    \node [bvertex](x1) at (1,-0.5) {};
    \node [bvertex](x2) at (1.5,-0.5) {};
    \node [bvertex](x3) at (2,-0.5) {};
    \node [bvertex](x4) at (2.5,-0.5) {};
    \node [bvertex](x5) at (3,-0.5) {};
    \node [bvertex](x6) at (3.5,-0.5) {};
    \node [bvertex, label={[font=\normalsize] right:$v'$}](x7) at (4.0,-0.5) {};

    \node [bvertex] (y1) at (-1.25,0) {};
    \node [bvertex] (y2) at (-1.75,0) {};

    \node [bvertex] (z1) at (1.25,0) {};
    \node [bvertex] (z2) at (1.75,0) {};

    \draw[-] (v7) .. controls (-1,2.4) and (1,2.4) .. (w7);
    
    \draw[-] (u7) .. controls (-1,-2.4) and (1,-2.4) .. (x7);
    
    \draw (a) -- (b);
    \draw (a) -- (v1) -- (v2) -- (v3) -- (v4) -- (v5);
    \draw (a) -- (u1) -- (u2) -- (u3) -- (u4) -- (u5);

    \draw (b) -- (w1) -- (w2)  -- (w4) -- (w5);
     \draw (b) -- (x1) -- (x2) -- (x3) -- (x4) -- (x5);
     \draw (w1) -- (x1) -- (z1) -- (x2) -- (z2) -- (x3) -- (w4);
     \draw (x3) -- (w4) -- (x4) -- (w5) -- (x5);
       \draw (v1) -- (u1) -- (y1) -- (u2) -- (y2) -- (u3) -- (v3) -- (u4) -- (v4) -- (u5) -- (v5);
       \draw (w1) -- (z1) -- (w2) -- (z2) -- (w4);
       \draw (v1) -- (y1) -- (v2) -- (y2) -- (v3);
       \draw (y1) -- (y2);
       \draw (z1) -- (z2);
       \draw (c)--(a);  \draw (c)--(b);  \draw (c)--(v1);  \draw (c)--(v2); \draw (c)--(v3); \draw (c)--(v4);  \draw (c)--(v5);\draw (c)--(v6);  \draw (c)--(v7);
       \draw (c)--(w1);  \draw (c)--(w2); \draw (c)--(w4);  \draw (c)--(w5); \draw (c)--(w6);  \draw (c)--(w7);

     \draw (d)--(a);  \draw (d)--(b);  \draw (d)--(u1);  \draw (d)--(u2); \draw (d)--(u3); \draw (d)--(u4);  \draw (d)--(u5); \draw (d)--(u6);  \draw (d)--(u7);
     \draw (d)--(x1);  \draw (d)--(x2); \draw (d)--(x3); \draw (d)--(x4);  \draw (d)--(x5); \draw (d)--(x6);  \draw (d)--(x7);

     \draw (v5)--(v6) -- (v7) -- (u7) -- (u6) -- (v6) -- (u7); \draw (u6) --(u5);

     \draw (w5)-- (w6) -- (w7) -- (x7) -- (x6);
     \draw (w6)--(x5);
     \draw (w7) -- (x6) -- (x5);

     \node (I1) at (-4.5,-0.5) {};
     \node (I2) at (4,0.5){};
     \draw (v7) -- (I1);\draw (I2) -- (x7);
        
\end{tikzpicture}

%% file: diamond.tikz
\begin{tikzpicture}[
    scale=1,
    wvertex/.style={circle, draw=red,   fill=red,   scale=0.3},
    bvertex/.style={circle, draw=black, fill=black, scale=0.3},
    rvertex/.style={circle, draw=red,   fill=red,   scale=0.3},
    sbvertex/.style={circle, draw=black, fill=black, scale=0.1}
  ]
  \node[rvertex] (T) at (2.8,2) {};
  \node[rvertex] (B) at (2.8,-2) {};

  \foreach \i in {0,...,7} {
    \pgfmathsetmacro{\x}{0.8*\i}
    \node[bvertex] (U\i) at (\x, 0.4)  {};
    \node[bvertex] (L\i) at (\x,-0.4) {};
  }

  \foreach \i in {0,...,7}
    \draw (T) -- (U\i);

  \foreach \i in {0,...,7}
    \draw (L\i) -- (B);

  \foreach \i in {0,...,6} {
    \pgfmathtruncatemacro{\j}{\i+1}
    \draw (U\i) -- (U\j);
    \draw (L\i) -- (L\j);
  }

  \foreach \i in {0,...,7}
    \draw (U\i) -- (L\i);

  \foreach \i in {0,...,6} {
    \pgfmathtruncatemacro{\j}{\i+1}
    \draw (U\j) -- (L\i);
  }

\end{tikzpicture}

%% file: blow_up.tikz
\tikzset{
  pics/fanGrid/.style={
    code={
      \def\cell{0.4}   
      \def\half{0.2}   

      \node[red,circle,fill,scale=0.3] (T) at (2*\cell,  2) {};
      \node[red,circle,fill,scale=0.3]          at (2*\cell, -2) {};

      \foreach \i in {0,...,4} {
        \pgfmathsetmacro\x{\i*\cell}
        \node[black,circle,fill,scale=0.3] at (\x,  \half) {};
        \node[black,circle,fill,scale=0.3] at (\x, -\half) {};
        \draw (T) -- (\x,\half)  (\x,-\half) -- (2*\cell,-2);
        \draw (\x,\half) -- (\x,-\half);
        \ifnum\i<4
          \pgfmathtruncatemacro{\j}{\i+1}
          \draw (\x,\half) -- ({\j*\cell},\half)
                (\x,-\half) -- ({\j*\cell},-\half)
                ({\j*\cell},\half) -- (\x,-\half);
        \fi
      }
    }
  }
}

\begin{tikzpicture}
    \pic at (2,0.8)[rotate=-90] {fanGrid};

    \pic at (3.2,2)[rotate=0] {fanGrid};

    \pic at (2,3.2)[rotate=90] {fanGrid};

    \pic at (0.8,2)[rotate=180] {fanGrid};

\end{tikzpicture}

%% file: blow-up2.tikz
\tikzset{
  pics/fanGrid/.style={
    code={
      \def\cell{0.4}   
      \def\half{0.2}   

      \node[red,circle,fill,scale=0.3] (T) at (2*\cell,  2) {};
      \node[red,circle,fill,scale=0.3]          at (2*\cell, -2) {};

      \foreach \i in {0,...,4} {
        \pgfmathsetmacro\x{\i*\cell}
        \node[black,circle,fill,scale=0.3] at (\x,  \half) {};
        \node[black,circle,fill,scale=0.3] at (\x, -\half) {};
        \draw (T) -- (\x,\half)  (\x,-\half) -- (2*\cell,-2);
        \draw (\x,\half) -- (\x,-\half);
        \ifnum\i<4
          \pgfmathtruncatemacro{\j}{\i+1}
          \draw (\x,\half) -- ({\j*\cell},\half)
                (\x,-\half) -- ({\j*\cell},-\half)
                ({\j*\cell},\half) -- (\x,-\half);
        \fi
      }
    }
  }
}

\begin{tikzpicture}
    \pic at (2,0.8)[rotate=-90] {fanGrid};

    \pic at (3.2,2)[rotate=0] {fanGrid};

    \pic at (2,3.2)[rotate=90] {fanGrid};

    \pic at (0.8,2)[rotate=180] {fanGrid};

\draw (1.8,0.8)-- (0.8,1.8);
\draw (0.8,2.2) -- (1.8,3.2);
\draw (2.2,3.2) -- (3.2,2.2);
\draw (3.2,1.8) -- (2.2,0.8);
\draw (0.8,1.8) -- (2.2,0.8)--(3.2,2.2)--(1.8,3.2)--(0.8,1.8)--(3.2,2.2);

\end{tikzpicture}

%% file: G11.tikz
\begin{tikzpicture}[scale=1, wvertex/.style={circle, draw=red, fill=red, scale=0.2}, bvertex/.style={circle, draw=black, fill=black, scale=0.2},rvertex/.style={circle, draw=red, fill=red, scale=0.2}, sbvertex/.style={circle, draw=black, fill=black, scale=0.1}] 

\foreach \i in {1,...,5} {
  \node [bvertex] (A\i) at (72*\i+18:1.5) {}; 
}

\draw (A1) -- (A2) -- (A3) -- (A4) -- (A5) -- (A1);

\foreach \i in {1,...,5} {
  \node [bvertex] (B\i) at (72*\i+18+180:0.7) {}; 
}

\draw (B1) -- (B2) -- (B3) -- (B4) -- (B5) -- (B1);

\node [bvertex] (C) at (0,0) {};

\draw (C) -- (B1);
\draw (C) -- (B2);
\draw (C) -- (B3);
\draw (C) -- (B4);
\draw (C) -- (B5);

\draw (A3) -- (B1) --(A4) -- (B2) -- (A5) --(B3) --(A1) -- (B4) --(A2) --(B5) --(A3);

\end{tikzpicture}

%% file: Ci.tikz
\begin{tikzpicture}[scale=1, wvertex/.style={circle, draw=red, fill=red, scale=0.2}, bvertex/.style={circle, draw=black, fill=black, scale=0.2},rvertex/.style={circle, draw=red, fill=red, scale=0.2}, sbvertex/.style={circle, draw=black, fill=black, scale=0.1}] 

\foreach \i in {1,...,5} {
  \node [bvertex] (A\i) at (72*\i+18:1.5) {}; 
}

\node at ($(A1)+ (0,-0.5)$) {$x_{i-1}$};
\node at ($(A2)+ (-0.4,0)$) {$d_{i-1}$};
\node at ($(A3)+ (-0.4,-0.1)$) {$c_{i-1}$};
\node at ($(A4)+ (0.4,-0.1)$) {$b_{i-1}$};
\node at ($(A5)+ (0.4,0)$) {$a_{i-1}$};

  \node [bvertex] (u1) at ($(A1)+(0,0.8)$) {};
  \node at ($(u1)+(0.3,0.1)$) {$u_{i}$};

  \draw (u1)--(A1);
  \draw (u1)--(A2);
  \draw (u1)--(A5);

\draw (A1) -- (A2) -- (A3) -- (A4) -- (A5) -- (A1);

\node [sbvertex] at (0,0.1){};
\node [sbvertex] at (0.2,0.1){};
\node [sbvertex] at (-0.2,0.1){};
\end{tikzpicture}

%% file: H2_case1.tikz
\begin{tikzpicture}[scale=1, wvertex/.style={circle, draw=red, fill=red, scale=0.2}, bvertex/.style={circle, draw=black, fill=black, scale=0.2},rvertex/.style={circle, draw=red, fill=red, scale=0.2}, sbvertex/.style={circle, draw=black, fill=black, scale=0.1}] 

\foreach \i in {1,...,5} {
  \node [bvertex] (A\i) at (72*\i+18:1.5) {}; 
}

\node at ($(A1)+ (0.3,0.1)$) {$x_0$};
\node at ($(A2)+ (-0.2,0.25)$) {$d_0$};
\node at ($(A3)+ (-0.3,-0.1)$) {$c_0$};
\node at ($(A4)+ (0.3,-0.1)$) {$b_0$};
\node at ($(A5)+ (0.2,0.25)$) {$a_0$};

  \node [bvertex] (u1) at ($(A1)+(0,1)$) {};
  \node at ($(u1)+(0.3,0.1)$) {$u_1$};

  \draw (u1)--(A1);
  \draw (u1)--(A2);
  \draw (u1)--(A5);

  \node [bvertex] (u2) at (72*5+18:2.5) {};
  \node at ($(u2)+(0.2,0.2)$) {$u_2$};
  \draw (u2)--(u1);
  \draw (u2)--(A5);
  \draw (u2)--(A4);
\draw (A1) -- (A2) -- (A3) -- (A4) -- (A5) -- (A1);

\foreach \i in {1,...,5} {
  \node [bvertex] (B\i) at (72*\i+18+180:0.7) {}; 
}

\draw (B1) -- (B2) -- (B3) -- (B4) -- (B5) -- (B1);

\node [bvertex] (C) at (0,0) {};

\draw (C) -- (B1);
\draw (C) -- (B2);
\draw (C) -- (B3);
\draw (C) -- (B4);
\draw (C) -- (B5);

\draw (A3) -- (B1) --(A4) -- (B2) -- (A5) --(B3) --(A1) -- (B4) --(A2) --(B5) --(A3);

\end{tikzpicture}

%% file: H2_case2.tikz
\begin{tikzpicture}[scale=1, wvertex/.style={circle, draw=red, fill=red, scale=0.2}, bvertex/.style={circle, draw=black, fill=black, scale=0.2},rvertex/.style={circle, draw=red, fill=red, scale=0.2}, sbvertex/.style={circle, draw=black, fill=black, scale=0.1}] 

\foreach \i in {1,...,5} {
  \node [bvertex] (A\i) at (72*\i+18:1.5) {}; 
}

\node at ($(A1)+ (0.3,0.1)$) {$x_0$};
\node at ($(A2)+ (-0.2,0.25)$) {$d_0$};
\node at ($(A3)+ (-0.3,-0.1)$) {$c_0$};
\node at ($(A4)+ (0.3,-0.05)$) {$b_0$};
\node at ($(A5)+ (0.2,0.25)$) {$a_0$};

  \node [bvertex] (u1) at ($(A1)+(0,1)$) {};
  \node at ($(u1)+(0.3,0.1)$) {$u_1$};

  \draw (u1)--(A1);
  \draw (u1)--(A2);
  \draw (u1)--(A5);

  \node [bvertex] (u2) at ($(A4)+(0.5,-0.5)$) {};
  \node at ($(u2)+(0.2,-0.2)$) {$u_2$};
  \draw (u2)--(A3);
  \draw (u2)--(A5);
  \draw (u2)--(A4);
\draw (A1) -- (A2) -- (A3) -- (A4) -- (A5) -- (A1);

\foreach \i in {1,...,5} {
  \node [bvertex] (B\i) at (72*\i+18+180:0.7) {}; 
}

\draw (B1) -- (B2) -- (B3) -- (B4) -- (B5) -- (B1);

\node [bvertex] (C) at (0,0) {};

\draw (C) -- (B1);
\draw (C) -- (B2);
\draw (C) -- (B3);
\draw (C) -- (B4);
\draw (C) -- (B5);

\draw (A3) -- (B1) --(A4) -- (B2) -- (A5) --(B3) --(A1) -- (B4) --(A2) --(B5) --(A3);

\end{tikzpicture}

%% file: Hp_case1.tikz
\begin{tikzpicture}[scale=1, wvertex/.style={circle, draw=red, fill=red, scale=0.2}, bvertex/.style={circle, draw=black, fill=black, scale=0.2},rvertex/.style={circle, draw=red, fill=red, scale=0.2}, sbvertex/.style={circle, draw=black, fill=black, scale=0.1}] 

\foreach \i in {1,...,5} {
  \node [bvertex] (A\i) at (72*\i+18:1.5) {}; 
}

\node at ($(A1)+ (0.3,0.1)$) {$x_p$};
\node at ($(A2)+ (-0.2,0.25)$) {$d_p$};
\node at ($(A3)+ (-0.3,-0.1)$) {$c_p$};
\node at ($(A4)+ (0.3,-0.1)$) {$b_p$};
\node at ($(A5)+ (0.7,0.25)$) {$a_p(=u_p)$};

  \node [bvertex] (u1) at ($(A1)+(0,0.8)$) {};
  \node at ($(u1)+(0.3,0.25)$) {$u_{p+1}$};

  \draw (u1)--(A1);
  \draw (u1)--(A2);
  \draw (u1)--(A5);

  \node [bvertex] (u2) at (72*5+18:0.8) {};
  \node at ($(u2)+(-0.4,-0.2)$) {$x_{p-1}$};
  \draw (u2)--(A1);
  \draw (u2)--(A5);
  \draw (u2)--(A4);
\draw (A1) -- (A2) -- (A3) -- (A4) -- (A5) -- (A1);

\node [sbvertex] at (-0.5,0.1) {};
\node [sbvertex] at (-0.7,0.1) {};
\node [sbvertex] at (-0.3,0.1) {};
\end{tikzpicture}

%% file: Hp_case2.tikz
\begin{tikzpicture}[scale=1, wvertex/.style={circle, draw=red, fill=red, scale=0.2}, bvertex/.style={circle, draw=black, fill=black, scale=0.2},rvertex/.style={circle, draw=red, fill=red, scale=0.2}, sbvertex/.style={circle, draw=black, fill=black, scale=0.1}] 

\foreach \i in {1,...,5} {
  \node [bvertex] (A\i) at (72*\i+18:1.5) {}; 
}

\node at ($(A1)+ (0.3,0.1)$) {$x_p$};
\node at ($(A2)+ (-0.2,0.25)$) {$d_p$};
\node at ($(A3)+ (-0.3,-0.1)$) {$c_p$};
\node at ($(A4)+ (0.8,-0.1)$) {$b_p(=u_p)$};
\node at ($(A5)+ (0.2,0.25)$) {$a_p$};

  \node [bvertex] (u1) at ($(A1)+(0,0.8)$) {};
  \node at ($(u1)+(0.3,0.25)$) {$u_{p+1}$};

  \draw (u1)--(A1);
  \draw (u1)--(A2);
  \draw (u1)--(A5);

   \node [bvertex] (u2) at ($(A4)+(-0.5,0.5)$) {};
  \node at ($(u2)+(-0.3,0.2)$) {$x_{p-1}$};
  \draw (u2)--(A3);
  \draw (u2)--(A5);
  \draw (u2)--(A4);
\draw (A1) -- (A2) -- (A3) -- (A4) -- (A5) -- (A1);

\node [sbvertex] at (0,0.2) {};
\node [sbvertex] at (-0.2,0.2) {};
\node [sbvertex] at (0.2,0.2) {};

\end{tikzpicture}

%% file: H3.tikz
\begin{tikzpicture}[scale=1, wvertex/.style={circle, draw=red, fill=red, scale=0.2}, bvertex/.style={circle, draw=black, fill=black, scale=0.2},rvertex/.style={circle, draw=red, fill=red, scale=0.2}, sbvertex/.style={circle, draw=black, fill=black, scale=0.1}] 

\foreach \i in {1,...,5} {
  \node [bvertex] (A\i) at (72*\i+18:1.5) {}; 
}

\node at ($(A1)+ (0.3,0.1)$) {$x_0$};
\node at ($(A2)+ (-0.2,0.25)$) {$d_0$};
\node at ($(A3)+ (-0.3,-0.05)$) {$c_0$};
\node at ($(A4)+ (0.3,-0.05)$) {$b_0$};
\node at ($(A5)+ (0.2,0.25)$) {$a_0$};

  \node [bvertex] (u1) at ($(A1)+(0,0.8)$) {};
  \node at ($(u1)+(0.3,0.1)$) {$u_1$};

  \draw (u1)--(A1);
  \draw (u1)--(A2);
  \draw (u1)--(A5);

  \node [bvertex] (u2) at ($(A4)+(0.5,-0.5)$) {};
  \node at ($(u2)+(0.2,-0.2)$) {$u_2$};
  \draw (u2)--(A3);
  \draw (u2)--(A5);
  \draw (u2)--(A4);

\node [bvertex] (u3) at ($(A3)+(-0.5,-0.5)$) {};
  \node at ($(u3)+(-0.2,-0.2)$) {$u_3$};
  \draw (u3)--(A3);
  \draw (u3)--(A2);
  \draw (u3)--(u2);
  
\draw (A1) -- (A2) -- (A3) -- (A4) -- (A5) -- (A1);

\foreach \i in {1,...,5} {
  \node [bvertex] (B\i) at (72*\i+18+180:0.7) {}; 
}

\draw (B1) -- (B2) -- (B3) -- (B4) -- (B5) -- (B1);

\node [bvertex] (C) at (0,0) {};

\draw (C) -- (B1);
\draw (C) -- (B2);
\draw (C) -- (B3);
\draw (C) -- (B4);
\draw (C) -- (B5);

\draw (A3) -- (B1) --(A4) -- (B2) -- (A5) --(B3) --(A1) -- (B4) --(A2) --(B5) --(A3);

\end{tikzpicture}

%% file: H4.tikz
\begin{tikzpicture}[scale=1, wvertex/.style={circle, draw=red, fill=red, scale=0.2}, bvertex/.style={circle, draw=black, fill=black, scale=0.2},rvertex/.style={circle, draw=red, fill=red, scale=0.2}, sbvertex/.style={circle, draw=black, fill=black, scale=0.1}] 

\foreach \i in {1,...,5} {
  \node [bvertex] (A\i) at (72*\i+18:1.5) {}; 
}

\node at ($(A1)+ (0.3,0.1)$) {$x_0$};
\node at ($(A2)+ (-0.2,0.25)$) {$d_0$};
\node at ($(A3)+ (-0.3,-0.05)$) {$c_0$};
\node at ($(A4)+ (0.3,-0.05)$) {$b_0$};
\node at ($(A5)+ (0.2,0.25)$) {$a_0$};

  \node [bvertex] (u1) at ($(A1)+(0,0.8)$) {};
  \node at ($(u1)+(0.3,0.1)$) {$u_1$};

  \draw (u1)--(A1);
  \draw (u1)--(A2);
  \draw (u1)--(A5);

  \node [bvertex] (u2) at ($(A4)+(0.5,-0.5)$) {};
  \node at ($(u2)+(0.27,-0.05)$) {$u_2$};
  \draw (u2)--(A3);
  \draw (u2)--(A5);
  \draw (u2)--(A4);

\node [bvertex] (u3) at ($(A3)+(-0.5,-0.5)$) {};
  \node at ($(u3)+(-0.2,-0.2)$) {$u_3$};
  \draw (u3)--(A3);
  \draw (u3)--(A2);
  \draw (u3)--(u2);

   \node [bvertex] (u4) at ($(u2)+(0.5,-0.5)$) {};
  \node at ($(u4)+(0.2,-0.2)$) {$u_4$};
  \draw (u4)--(A5);
  \draw (u4)--(u2);
  \draw (u4)--(u3); 
  
\draw (A1) -- (A2) -- (A3) -- (A4) -- (A5) -- (A1);

\foreach \i in {1,...,5} {
  \node [bvertex] (B\i) at (72*\i+18+180:0.7) {}; 
}

\draw (B1) -- (B2) -- (B3) -- (B4) -- (B5) -- (B1);

\node [bvertex] (C) at (0,0) {};

\draw (C) -- (B1);
\draw (C) -- (B2);
\draw (C) -- (B3);
\draw (C) -- (B4);
\draw (C) -- (B5);

\draw (A3) -- (B1) --(A4) -- (B2) -- (A5) --(B3) --(A1) -- (B4) --(A2) --(B5) --(A3);

\end{tikzpicture}